\documentclass{article}
\usepackage{amsmath,amscd}
\usepackage{amssymb}
\usepackage{theorem}
\usepackage{gastex}
\usepackage{url}
\usepackage[dvipsnames]{xcolor}
\newtheorem{theorem}{Theorem}[section]
\newtheorem{proposition}[theorem]{Proposition}
\newtheorem{corollary}[theorem]{Corollary}
\newtheorem{lemma}[theorem]{Lemma}
\usepackage{todonotes}
\usetikzlibrary{shapes,snakes}

{\theorembodyfont{\rmfamily}%
  \newtheorem{example}[theorem]{Example}
   }
\newenvironment{proof}{\noindent\textit{Proof.}}
{\QED\vskip\theorempostskipamount} 
\newenvironment{proofof}[1]{\noindent\textit{Proof \protect{#1}.}}
{\QED\vskip\theorempostskipamount}
\def\petitcarre{\vrule height4pt width 4pt depth0pt}
\def\QED{\relax\ifmmode\eqno{\hbox{\petitcarre}}\else{%
  \unskip\nobreak\hfil\penalty50\hskip2em\hbox{}\nobreak\hfil
  \petitcarre
  \parfillskip=0pt \finalhyphendemerits=0\par\smallskip}
  \fi}

\newcommand{\Z}{\mathbb{Z}}
\newcommand{\N}{\mathbb{N}}

\DeclareMathOperator{\Card}{Card}

\def\u(#1){\underline{#1\!}\,}

\def\1{\mathbf{1}}

\newcommand{\RR}{\mathcal R}
\newcommand{\LL}{\mathcal L}

\newcommand{\E}{\mathcal E}
\newcommand{\CR}{\mathcal{CR}}
\numberwithin{equation}{section}

    \def\soft#1{\leavevmode\setbox0=\hbox{h}\dimen7=\ht0\advance
    \dimen7 by-1ex\relax\if t#1\relax\rlap{\raise.6\dimen7
    \hbox{\kern.3ex\char'47}}#1\relax\else\if T#1\relax
    \rlap{\raise.5\dimen7\hbox{\kern1.3ex\char'47}}#1\relax
    \else\if d#1\relax\rlap{\raise.5\dimen7\hbox{\kern.9ex
    \char'47}}#1\relax\else\if D#1\relax\rlap{\raise.5\dimen7
    \hbox{\kern1.4ex\char'47}}#1\relax\else\if l#1\relax
    \rlap{\raise.5\dimen7\hbox{\kern.4ex\char'47}}#1\relax
    \else\if L#1\relax\rlap{\raise.5\dimen7\hbox{\kern.7ex
    \char'47}}#1\relax\else\message{accent \string\soft
    \space #1 not defined!}#1\relax\fi\fi\fi\fi\fi\fi} 

\title{Eventually dendric shift spaces}
\author{Francesco Dolce and Dominique Perrin}

\begin{document}

\maketitle
\tableofcontents

\begin{abstract}
We define a new class of shift spaces which contains a number of classes of interest, like Sturmian shifts used in discrete geometry.
We show that this class is closed under two natural transformations.
The first one is called conjugacy and is obtained by sliding block coding.
The second one is called the complete bifix decoding, and typically includes codings by non overlapping blocks of fixed length.
\end{abstract}

\section{Introduction}
Shift spaces are the sets of two-sided infinite words avoiding the words of a given language $F$ denoted $X_F$.
In this way the traditional hierarchy of classes of languages translates into a hierarchy of shift spaces.
The shift space $X_F$ is called of finite type when one starts with a finite language $F$ and sofic when one starts with a regular language $F$.

There is a natural equivalence between shift spaces called conjugacy.
Two shift spaces are conjugate if there is a sliding block coding sending bijectively one upon the other (in this case the inverse map has the same form).
Many basic questions are still open concerning conjugacy.
For example, it is surprisingly not known whether the conjugacy of shifts of finite type is decidable.

The complexity of a shift space $X$ is the function $n\mapsto p(n)$ where $p(n)$ is the number of admissible blocks of length $n$ in $X$.
The complexities of conjugate shifts of linear complexity have the same growth rate.

In this paper, we are interested in shift spaces of at most linear complexity.
This class is important for many reasons and includes the class of Sturmian shifts which are by definition those of complexity $n+1$, which play a role as binary codings of discrete lines.
Several books are devoted to the study of such shifts (see~\cite{PytheasFogg2002} or~\cite{Queffelec2010} for example).
We define a new class of shifts of at most linear complexity, called eventually dendric, extends the class of dendric shifts introduced in~\cite{BertheDeFeliceDolceLeroyPerrinReutenauerRindone2015} (under the name of \emph{tree sets} given to their language) which themselves extend naturally episturmian shifts (also
called Arnoux-Rauzy shifts) and interval exchange shifts.

Our first main result is that this class is closed under conjugacy.
We also prove that it is closed under a second transformation, namely complete bifix decoding, which is important because it includes coding by non overlapping blocks of fixed length.
These two results show the robustness of the class of eventually dendric shifts, giving a strong motivation for its introduction.

The class of dendric shifts (defined below) is known to be closed under complete bifix decoding (see~\cite{BertheDeFeliceDolceLeroyPerrinReutenauerRindone2013m}) but it is not closed under conjugacy.
This fact was the initial motivation for introducing eventually dendric shifts, following a suggestion of Fabien Durand.

We now describe the results in some more detail.

A dendric shift $X$ is defined by introducing the extension graph of a word in the language $\LL(X)$ of $X$ and by requiring that this graph is a tree for every word in $\LL(X)$.
It has many interesting properties which involve free groups.
In particular, in a dendric shift $X$ on the alphabet $A$, the group generated by the set of return words to some word in $\LL(X)$ is the free group on the alphabet and, in particular, has $\Card{A}$ free generators.
This generalizes a property known for Sturmian shifts whose link with automorphisms of the free group was noted by Arnoux and Rauzy.

The class of eventually dendric shifts, introduced in this paper, is defined by the property that the extension graph of every word $w$ in the language of the shift is a tree for every long enough word $w$.

Our main results are that the class of eventually dendric shifts is closed under
\begin{itemize}
	\item conjugacy (Theorem~\ref{theoremConjugacy}), and
	\item complete bifix decoding (Theorem~\ref{theoremCompleteBifixDecoding}).
\end{itemize}

The paper is organized as follows.
In the first section, we introduce the definition of the extension graph and of an eventually dendric shift.
In Section~\ref{sectionComplexity}, we recall some mostly known properties on the complexity of a shift and of special words.
We prove a result which characterizes eventually dendric shifts by the extension properties of special words (Proposition~\ref{propositionSG}).
In Section~\ref{sectionAsymptotic}, we use the classical notion of asymptotic equivalence to give a second characterization of eventually dendric shifts (Proposition~\ref{propositionCaracterizationDendric}).
In Section~\ref{sectionSimpletrees}, we introduce the notion of a simple tree and we prove that for an eventually dendric shift, the extension graph of every long enough word is a simple tree (Proposition~\ref{propositionExt}), a property which holds trivially for every word in a Sturmian shift.
In Section~\ref{sectionConjugacy} we prove the first of our main results (Theorem~\ref{theoremConjugacy}).
In the next sections (Section~\ref{sectionProperty} to \ref{sectionBifixDecoding}), we prove additional properties of eventually dendric shifts.
We first prove that eventually dendric shifts are minimal as soon as they are irreducible (Theorem~\ref{theoremMinimalDendric}), a property already known for dendric shifts \cite{DolcePerrin2017}.
Next we introduce generalized extension graphs in which extension by words of of fixed length replace extension by letters.
We prove that one obtain an equivalent definition of eventually dendric shifts using these generalized extension graphs (Theorem~\ref{theoremGeneralized}).
Finally, we prove that the class of eventually dendric shifts is closed under complete bifix decoding, a result already known for dendric shifts.

\paragraph{Acknowledgements}
We thank Val\'erie Berth\'e, Paulina Cecchi, Fabien Durand and Samuel Petite for useful conversations on this subject and especially Fabien Durand which suggested to us the statement of Theorem~\ref{theoremConjugacy}.

\section{Eventually dendric shifts}
Let $A$ be a finite alphabet.
We consider the set$A^Z$ of bi-infinite words  on $A$ as a topological space for the product topology.
The \emph{shift map} $\sigma_A : A^\Z \to A^\Z$ is defined by $y=\sigma_A(x)$ if $y_i=x_{i+1}$ for every $i \in \Z$.
It is a one-to-one continuous map.

We also consider the topological space $A^\N$ of one-sided infinite words.
We still denote by $\sigma_A$ the map from $A^\N$ to $A^\N$ defined by $\sigma_A(x)=y$ if $y_i=x_{i+1}$ for all $i\in \N$.
Note that $\sigma_A$ is not one-to-one as soon as $\Card(A) \ge 2$.

A \emph{shift space} on the alphabet $A$ is a subset $X$ of the set $A^\Z$ which is closed and invariant under the shift, that is such that $\sigma_A(X)=X$ (for more on shift spaces see, for instance,~\cite{PytheasFogg2002}).

We denote by $\LL(X)$ the language of $X$, which is the set of finite factors of the elements of $X$.
A language $\LL$ on the alphabet $A$  is the language of a shift if and only if it is \emph{factorial} (that is contains the factors of its elements) and \emph{extendable} (that is for any $w \in \LL$ there are letters $a,b \in A$ such that $awb \in \LL$).

For $n \ge 0$ we denote
\begin{eqnarray*}
\LL_n(X) & = & \LL(X) \cap A^n \\
\LL_{\ge n}(X) & = & \cup_{m\ge n} \LL_m(X).
\end{eqnarray*}

For $w \in \LL(X)$ and $n \ge 1$, we denote
\begin{eqnarray*}
L_n(w,X) & = & \{ u \in \LL_n(X) \mid uw \in \LL(X) \} \\
R_n(w,X) & = & \{ v \in \LL_n(X) \mid wv \in \LL(X) \} \\
E_n(w,X) & = & \{ (u,v) \in L_n(w,X) \times R_n(w,X) \mid uwv \in \LL(X) \}
\end{eqnarray*}

The \emph{extension graph} of order $n$ of $w$, denoted $\E_n(w,X)$, is the undirected bipartite graph with set of vertices the disjoint union of $L_n(w,X)$ and $R_n(w,X)$ and with edges the elements of $E_n(w,X)$.

When the context is clear, we denote $L_n(w), R_n(w), E_n(w)$ and $\E_n(w)$ instead of $L_n(w,X), R_n(w,X), E_n(w,X)$ and $\E_n(w,X)$.

A path in an undirected graph is \emph{reduced} if it does not contain successive equal edges.
For any $w \in \LL(X)$, since any vertex of $L_n(w)$ is connected to at least one vertex of $R_n(w)$, the bipartite graph $\E_n(w)$ is a tree if and only if there is a unique reduced path between every pair of vertices of $L_n(w)$ (resp. $R_n(w)$).

The shift $X$ is said to be \emph{eventually dendric} with threshold $m \ge 0$ if $\E_1(w)$ is a tree for every word $w \in \LL_{\ge m}(X)$.
It is said to be \emph{dendric} if we can choose $m = 0$.

The languages of dendric shifts were introduced in~\cite{BertheDeFeliceDolceLeroyPerrinReutenauerRindone2015} under the name of tree sets.
An important example of dendric shifts is formed by \emph{episturmian shifts} (also called Arnoux-Rauzy shifts), which are by definition such that $\LL(X)$ is closed by reversal and such that for every $n$ there exists a unique $w_n \in \LL_n(X)$ such that $\Card(R_1(w_n)) = \Card(A)$ and such that for every $w \in \LL_n(X) \setminus \{ w_n \}$ one has $\Card(R_1(w)) = 1$ (see~\cite{BertheDeFeliceDolceLeroyPerrinReutenauerRindone2015}).

\begin{example}
\label{exampleFibo}
Let $X$ be the \emph{Fibonacci shift}, which is generated by the morphism $a \mapsto ab, b \mapsto a$.
It is well known that it is a Sturmian shift (see~\cite{PytheasFogg2002}).
The graph $\E_1(a)$ is shown in Figure~\ref{figureFibo} on the left.
The graph $\E_3(a)$ is shown on the right.

	\begin{figure}[hbt]
	\centering
	\tikzset{node/.style={rectangle,draw,rounded corners=1.2ex}}
	\begin{tikzpicture}
	\node[node](a1al) {$a$};
	\node[node](a1bl) [below= 0.3cm of a1al] {$b$};
	\node[node](a1ar) [right= 1.5cm of a1al] {$a$};
	\node[node](a1br) [below= 0.3cm of a1ar] {$b$};
	\path[draw,thick, shorten <=0 -1pt, shorten >=-1pt]
	(a1al) edge node {} (a1br)
	(a1bl) edge node {} (a1ar);
	\path[draw,thick, shorten <=0 pt, shorten >=-0pt]
	(a1bl) edge node {} (a1br);
	\node[node](a3abal) [above right= 0cm and 2cm of a1ar] {$aba$};
	\node[node](a3aabl) [below= 0.3cm of a3abal] {$aab$};
	\node[node](a3babl) [below= 0.3cm of a3aabl] {$bab$};
	\node[node](a3babr) [right= 1.5cm of a3abal] {$bab$};
	\node[node](a3baar) [below= 0.3cm of a3babr] {$baa$};
	\node[node](a3abar) [below= 0.3cm of a3baar] {$aba$};
	\path[draw,thick]
	(a3abal) edge node {} (a3babr)
	(a3aabl) edge node {} (a3baar)
	(a3babl) edge node {} (a3abar);
	\path[draw,thick, shorten <=0 -2pt, shorten >=-2pt]
	(a3abal) edge node {} (a3baar)
	(a3aabl) edge node {} (a3abar);
	\end{tikzpicture}

 \caption{The graphs $\E_1(a)$ and $\E_3(a)$.}
 \label{figureFibo}
\end{figure}
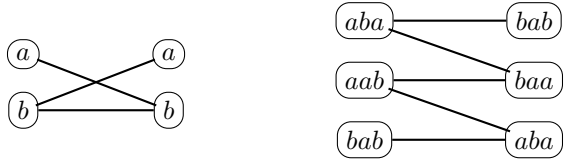
\end{example}

The tree sets of characteristic $c \ge 1$ introduced in~\cite{BertheDeFeliceDelecroixDolceLeroyPerrinRindone2017,DolcePerrin2017} give an example of eventually dendric shifts.

\begin{example}
\label{exampleJulien}
Let $X$ be the shift generated by the morphism $a \mapsto ab, b \mapsto cda, c \mapsto cd, d\mapsto abc$.
Its language is a tree set of characteristic $2$ (\cite[Example 4.2]{BertheDeFeliceDelecroixDolceLeroyPerrinRindone2017}) and it is actually a specular set.
The extension graph $\E_1(\varepsilon)$ is shown in Figure~\ref{figureSpecular}.

\begin{figure}[hbt]
	\centering
	\tikzset{node/.style={rectangle,draw,rounded corners=1.4ex}}
	\begin{tikzpicture}
	\node[node](eal) {$a$};
	\node[node](ebl) [below= 0.3cm of eal] {$b$};
	\node[node](ebr) [right= 1.5cm of eal] {$b$};
	\node[node](ecr) [below= 0.3cm of ebr] {$c$};
	\path[draw,thick]
	(eal) edge node {} (ebr)
	(ebl) edge node {} (ecr);
	\path[draw,thick, shorten <=0 -1.5pt, shorten >=-1.5pt]
	(eal) edge node {} (ecr);
	\node[node](ecl) [below= 0.3 of ebl] {$c$};
	\node[node](edl) [below= 0.3cm of ecl] {$d$};
	\node[node](edr) [right= 1.5cm of ecl] {$d$};
	\node[node](ear) [below= 0.3cm of edr] {$a$};
	\path[draw,thick]
	(ecl) edge node {} (edr)
	(edl) edge node {} (ear);
	\path[draw,thick, shorten <=0 -1.5pt, shorten >=-1.5pt]
	(ecl) edge node {} (ear);
	\end{tikzpicture}
	\caption{The extension graph $\E_1(\varepsilon)$.}
	\label{figureSpecular}
\end{figure}
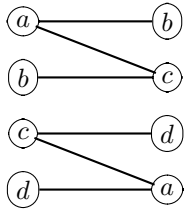

Since the extension graphs of all nonempty words are trees, the shift space is eventually dendric with threshold $1$.
\end{example}

\section{Complexity of shift spaces}
\label{sectionComplexity}
Let $X$ be a shift space.
For a word $w \in \LL(X)$ and $k\ge 1$, we denote
\begin{displaymath}
\ell_k(w) = \Card(L_k(w)), \
r_k(w) = \Card(R_k(w)), \
e_k(w) = \Card(E_k(w)).
\end{displaymath}
For any $w \in \LL(X)$, we have $1 \le \ell_k(w), r_k(w) \le e_k(w)$.
The word $w$ is \emph{left-$k$-special} if $\ell_k(w) > 1$, \emph{right-$k$-special} if $r_k(w) > 1$ and \emph{$k$-bispecial} if it is both left-$k$-special and right-$k$-special. For $k=1$, we use $\ell,e,r$ instead of $\ell_1,e_1,r_1$ and we simply say special instead of $k$-special.

We define the \emph{multiplicity} of $w$ as
\begin{displaymath}
m(w) = e(w)-\ell(w)-r(w)+1.
\end{displaymath}
We say that $w$ is \emph{strong} if $m(w) \ge 0$, \emph{weak} if $m(w) \le 0$ and \emph{neutral} if $m(w)=0$.

It is clear that
\begin{enumerate}
\item if $\E_1(w)$ is acyclic, then $w$ is weak,
\item if $\E_1(w)$ is connected, then $w$ is strong,
\item if $\E_1(w)$ is a tree, then $w$ is neutral.
\end{enumerate}

\begin{proposition}
\label{propositionNeutral}
Let $X$ be a shift space and let $w \in \LL(X)$.
If $w$ is neutral, then
\begin{equation}
\ell(w)-1 = \sum_{b \in R_1(w)}(\ell(wb)-1)
\label{eqgSum}
\end{equation}
\end{proposition}
\begin{proof}
Since $w$ is neutral, we have $e(w) = \ell(w)+r(w)-1$.
Thus
\begin{eqnarray*}
\sum_{b\in R_1(w)}(\ell(wb)-1) & = & e(w)-r(w) \\
 & = & \ell(w)-1.
\end{eqnarray*}
\end{proof}
Note that the symmetrical of Proposition~\ref{propositionNeutral} also holds: if $w \in \LL(X)$ is neutral then
$$
r(w) - 1 = \sum_{b \in L_1(w)} \left( r(bw) - 1 \right).
$$

Set further
\begin{eqnarray*}
p_n(X) & = & \Card(\LL_n(X)), \\
s_n(X) & = & p_{n+1}(X)-p_n(X), \\
b_n(X) & = & s_{n+1}(X)-s_n(X).
\end{eqnarray*}
The sequence $p_n(X)$ is called the \emph{complexity} of the shift space $X$.

The following result is from \cite{Cassaigne1997} (see also~\cite[Lemma 2.12]{BertheDeFeliceDolceLeroyPerrinReutenauerRindone2015}).
We include a proof for convenience.

\begin{proposition}
\label{propositionCassaigne}
We have for all $n \ge 0$,
\begin{equation}
s_n(X) = \sum_{w \in \LL_n(X)}(\ell(w)-1) = \sum_{w \in \LL_n(X)}(r(w)-1)
\label{eqs_n}
\end{equation}
and
\begin{equation}
b_n(X) = \sum_{w \in \LL_n(X)} m(w).
\label{eqb_n}
\end{equation}
In particular, the number of left-special (resp. right-special) words of length $n$ is bounded by $s_n(X)$.
\end{proposition}
\begin{proof}
We have
\begin{eqnarray*}
\sum_{w \in \LL_n(X)}(\ell(w)-1) & = & \sum_{w \in \LL_n(X)}\ell(w) - \Card(\LL_n(X)) \\
 & = & \Card(\LL_{n+1}(X)) - \Card(\LL_n(X)) = p_{n+1}-p_n \\
 & = & s_n(X)
\end{eqnarray*}
with the same result for $\sum_{w \in \LL_n(X)}(r(w)-1)$.
N
\begin{eqnarray*}
\sum_{w \in \LL_n(X)}m(w) & = & \sum_{w \in \LL_n(X)}(e(w)-\ell(w)-r(w)+1) \\
 & = & p_{n+2}(X)-2p_{n+1}(X)+p_n(X) = s_{n+1}(X)-s_n(X) = b_n(X).
\end{eqnarray*}
\end{proof}

We will use the following easy consequence of Proposition~\ref{propositionCassaigne}.
 
\begin{proposition}
\label{propositionComplexity}
Let $X$ be a shift space.
If $X$ is eventually dendric, then the sequence $s_n(X)$ is eventually constant.
\end{proposition}
\begin{proof}
Let $n \ge 1$ be such that the extension graph of every word in $\LL_{\ge n}(X)$ is a tree.
Then $b_m(X) = 0$ for every $m \ge n$.
Thus $s_m(X) = s_{m+1}(X)$ for every $m \ge n$, whence our conclusion.
\end{proof}

The converse of Proposition~\ref{propositionComplexity} is not true, as shown by the following example.

\begin{example}
Let $X$ be the \emph{Chacon ternary shift}, which is the substitutive shift space generated by the morphism $\varphi : a \mapsto aabc, b \mapsto bc, c\mapsto abc$.
It is well known that the complexity of $X$ is $p_n(X) = 2n+1$ and thus that $s_n = 2$ for all $n \ge 0$ (see~\cite[Section 5.5.2]{PytheasFogg2002}).
The extension graphs of $abc$ and $bca$ are shown in Figure~\ref{figureChacon}.

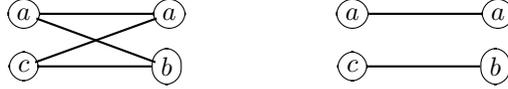
\begin{figure}[hbt]
\centering
	\tikzset{node/.style={rectangle,draw,rounded corners=1.4ex}}
	\begin{tikzpicture}
		\node[node](abcal) {$a$};
		\node[node](abccl) [below= 0.3cm of abcal] {$c$};
		\node[node](abcar) [right= 1.5cm of abcal] {$a$};
		\node[node](abcbr) [right= 1.5cm of abccl] {$b$};
		\path[draw,thick]
			(abcal) edge node {} (abcar)
			(abccl) edge node {} (abcbr);
		\path[draw,thick, shorten <=0 -1.5pt, shorten >=-1.5pt]
			(abcal) edge node {} (abcbr)
			(abccl) edge node {} (abcar);
		\node[node](bcaal) [right= 2 of abcar] {$a$};
		\node[node](bcacl) [below= 0.3cm of bcaal] {$c$};
		\node[node](bcaar) [right= 1.5cm of bcaal] {$a$};
		\node[node](bcabr) [right= 1.5cm of bcacl] {$b$};
		\path[draw,thick]
			(bcaal) edge node {} (bcaar)
			(bcacl) edge node {} (bcabr);
	\end{tikzpicture}

 \caption{The extension graphs of $abc$ and $bca$.}
 \label{figureChacon}
\end{figure}

Thus $m(abc) = 1$ and $m(bca) = -1$.
Let now $\alpha$ be the map on words defined by $\alpha(x) = abc \varphi(x)$.
Let us verify that if the extension graph of $x$ is the graph of Figure~\ref{figureChacon} on the left, the same holds for the extension graph of $y = \alpha(x)$.
Indeed, since $axa \in \LL(X)$, the word $\varphi(axa) = aabc \varphi(x) aabc = ayaabc$ is also in $\LL(X)$ and thus $(a,a) \in \E_1(y)$.
Since $cxa\in \LL(X)$ and since a letter $c$ is always preceded by a letter $b$, we have $bcxa \in \LL(X)$.
Thus $\varphi(bcxa) = bcyaabc \in \LL(X)$ and thus $(c,a) \in \E_1(y)$.
The proof of the other cases is similar.
The same property holds for a word $x$ with the extension graph on the right of Figure~\ref{figureChacon}.
This shows that there is an infinity of words whose extension graph is not a tree and thus the Chacon set is not eventually dendric.
\end{example}

Let $X$ be a shift space.
We define $LS_n(X)$ (resp. $LS_{\ge n}(X)$) as the set of left-special words of $\LL(X)$ of length $n$ (resp. at least $n$). We denote $LS(X)=\cup_{n\ge 1}LS_n(X)$.

The following result expresses the fact that eventually dendric shift spaces are characterized by an asymptotic property of left-special words which is a local version of the property defining Sturmian shift spaces.

\begin{proposition}
\label{propositionSG}
A shift space $X$ is eventually dendric if and only if there is an integer $n \ge 0$ such that any word $w$ of $LS_{\ge n}(X)$ has exactly one right extension $wb \in LS_{\ge n+1}(X)$ with $b \in A$ which is moreover such that $\ell(wb) = \ell(w)$.
\end{proposition}
\begin{proof}
Assume first that $X$ is eventually dendric with threshold $m$.
Then any word $w$ in $LS_{\ge m}(X)$ has at least one right extension in $LS(X)$.
Indeed, since $R_1(w)$ has at least two elements and since the graph $\E_1(w)$ is connected, there is at least one element of $R_1(w)$ which is connected by an edge to more than one element of $R_1(w)$.

Next, Equation~\eqref{eqgSum} shows that for any $w \in LS_{\ge m}(X)$ which has more than one right extension in $LS(X)$, one has $\ell(wb) < \ell(w)$ for each such extension. 
Thus the number of words in $LS_{\ge m}(X)$ prefix of one another which have more than one right extension is bounded by $\Card(A)$.
This proves that there exists an $n \ge m$ such that for any $w \in \LL_{\ge n}(X)$ there is exactly one $b \in A$ for which $wb \in LS(X)$. Moreover, one has then $\ell(wb) = \ell(w)$ by Equation~\eqref{eqgSum}.

Conversely, assume that the condition is satisfied for some integer $n$.
For any word $w$ in $\LL_{\ge n}(X)$, the graph $\E_1(w)$ is acyclic since all vertices in $R_1(w)$ except at most one have degree $1$.
Thus $w$ is weak.
Let $N$ be the length of $w$.
Then for every word $u$ of length $N$ and every $b \in R_1(u)$, 
one has $\ell(ub) = 1$ except for one letter $b$ 
such that $\ell(ub) = \ell(u)$.
Thus, by Proposition~\ref{propositionCassaigne},
\begin{displaymath}
s_N(X) = \sum_{u\in \LL_N(X)}(\ell(u)-1) = \sum_{v\in \LL_{N+1}(X)}(\ell(v)-1) = s_{N+1}(X).
\end{displaymath}
This shows that $b_N = 0$ for every $N \ge n$ and thus, by Proposition~\ref{propositionCassaigne} again, all words in $\LL_{\ge n}(X)$ are neutral.
Since all graphs $\E_1(w)$ are moreover acyclic, this forces that these graphs are trees and thus that $X$ is eventually dendric with threshold $n$.
\end{proof}

We give below an example of a shift space which is shown to be eventually dendric using Proposition~\ref{propositionSG}.

\begin{example}
\label{exampleImageTribo}
Let $X$ be the \emph{Tribonacci shift}, which is the episturmian shift space generated by the substitution $\varphi : a \mapsto ab, b \mapsto ac, c\mapsto a$ and let $\alpha$ be the morphism $\alpha : a \mapsto a, b \mapsto a, c \mapsto c$.
Let $\varphi^\omega(a)$ be the right infinite word having all $\varphi^n(a)$ for $n \ge 1$ as prefixes.
The left-special words for $X$ are the prefixes of $\varphi^\omega(a)$.
Indeed, it is easy to verify that if $w$ is left-special, then $\varphi(w)$ is also left-special.

Note that the set $\LL(X) \cap c \{a,b\}^* c$ is
\begin{displaymath}
\{ cabac, cabaabac, cababac \}.
\end{displaymath}
Since these three words are of distinct lengths, it follows that the restriction of $\alpha$ to the set $\LL(X) \cap c \{a,b\}^* c$ is injective.

Next we claim that the left-special words for $\alpha(X)$ containing a letter $c$ are the prefixes of $\alpha(\varphi^\omega(a))$ or $aa\alpha(\varphi^\omega(a))$ containing a letter $c$.
Indeed, if $w$ is a prefix of $\varphi^\omega(a)$, we have $\LL_1(w,X) = \{a,b,c\}$ and thus $\LL_1(\alpha(w),\alpha(X)) = \{a,c\}$ showing that $\alpha(w)$ is left-special.
Next, $\LL_3(w,X) = \{aba,bac,cab\}$ and thus $\LL_1(aa\alpha(w),\alpha(X)) = \{a,c\}$ showing t
hat $aa\alpha(w)$ is left-special.
Conversely, assume that $u$ is left-special for $\alpha(X)$ and contains a $c$.
Since $u$ is a prefix of a word ending with $c$, we may assume that $u$ ends with $c$.
Set $u = a^jcvc$ with $j \ge 0$.
By a previous remark, there is a unique word $s \in \LL(X)$ such that $csc \in \LL(X)$ and $\alpha(csc) = cvc$.
Since every word in $\LL(X)$ of length at least $7$ contains a $c$, we have $j \le 6$.
It is easy to verify by inspection of the possible left extensions of $c$ in $\LL(X)$ that $u$ is left-special only when $j=3$ or $j=5$ (see Figure~\ref{figureLeftExt}). 

\begin{figure}[hbt]
\centering
	\tikzset{title/.style={minimum size=0.5cm,inner sep=0pt}}
	\begin{tikzpicture}
		\node[title](caba) {$\cdots caba$};
		\node[title](cab) [below= 0.3cm of caba] {$\cdots cab$};
		\node[title](abc) [below= 0.3cm of cab] {$\cdots abc$};
		\node[title](abac) [right= 1.0cm of cab] {$abac$};
		\path[draw,thick]
			(caba) edge node {} (abac)
			(cab) edge node {} (abac)
			(abc) edge node {} (abac);
		\node[title](ca) [right= 3cm of caba] {$\cdots ca$};
		\node[title](aa) [right= 0.5cm of ca] {$aa$};
		\node[title](c) [below= 0.3cm of ca] {$\cdots c$};
		\node[title](aac) [below= 1cm of aa] {$\cdots aac$};
		\node[title](aaac) [right= 2.5cm of c] {$aaac$};
		\path[draw,thick]
			(ca) edge node {} (aa)
			(c) edge node {} (aa)
			(aa) edge node {} (aaac)
			(aac) edge node {} (aaac);
	\end{tikzpicture}

 \caption{The possible left extensions of $c$ in $\LL(X)$ and in $\alpha(\LL(X))$.}
 \label{figureLeftExt}
\end{figure}
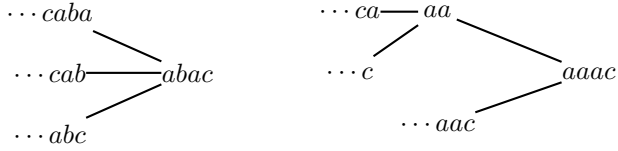

If $j=3$, then $u = \alpha(w)$ where $w = abacsc$ is left-special in $\LL(X)$ and thus is a prefix of $\varphi^\omega(a)$.
If $j=5$, then $u$ is the common image by $\alpha$ of $ababacsc$ and $baabacsc$.
Then $w=abacsc$ is left-special in $X$ and thus is a prefix of $\varphi^\omega(a)$.
Since $u = aa\alpha(w)$, the claim is proved.

It follows from the claim that the shift space $\alpha(X)$ satisfies the condition of Proposition~\ref{propositionSG} with $n=4$.
Thus we conclude that $\alpha(X)$ is dendric with threshold at most $4$.
The threshold is actually $4$ since $a^3$ has multiplicity $1$ in $\alpha(X)$.
\end{example}

\section{Asymptotic equivalence}
\label{sectionAsymptotic}

The \emph{orbit} of $x \in A^\Z$ is the equivalence class of $x$ under the action of the shift transformation.
Thus $y$ is in the orbit of $x$ if there is an $n \in \Z$ such that $x = \sigma_A^n(y)$.
We say that $x$ is a shift of $y$ if they belong to the same orbit.

For $x \in A^\Z$, denote
\begin{displaymath}
 x^- = \cdots x_{-2}x_{-1} \mbox{ and } x^+ = x_0x_1 \cdots
\end{displaymath} 
and $x = x^-\cdot x^+$.
When $X$ is a shift space, we denote $X^+$ the set of right infinite words $u$ such that $u=x^+$ for some $x \in X$.

A right infinite word $u \in A^\N$ is a \emph{tail} of the two-sided infinite word $x \in A^\Z$ if $u = y^+$ for some shift $y$ of $x$, that is $u = x_nx_{n+1} \cdots$ for some $n \in \Z$.

Let $X$ be a shift space on the alphabet $A$.
The \emph{right asymptotic equivalence} is the equivalence on $X$ defined as follows.
Two elements $x, y$ of $X$ are asymptotically equivalent if there exists two shifts $x',y'$ of $x,y$ such that $x'^+ = y'^+$.
In other words, $x,y$ are right asymptotic equivalent if they have a common tail (see Figure~\ref{figureAsymptoticFibo}).

\begin{figure}[hbt]
\centering
	\tikzset{title/.style={minimum size=-0cm,inner sep=0pt}}
	\tikzset{node/.style={circle,draw,minimum size=0.2cm,inner sep=0pt}}
	\begin{tikzpicture}
		\node[title](hh) {};
		\node[title](bb) [below= 1cm of hh] {};
		\node[title](h) [right= 1cm of hh] {};
		\node[title](hm) [below right= 0.4cm and 0.5cm of h] {};
		\node[title](b) [right= 1cm of bb] {};
		\node[title](bm) [above right= 0.4cm and 0.5cm of b] {};
		\node[title](d) [below right= 0.5cm and 1.5cm of hh] {};
		\node[node](dd) [right= 0.5cm of d] {};
		\node[title](dm) [right= 2cm of dd] {};
		\draw[line width=0.3mm, above, shorten >=-1pt] (hh) edge node {$x^-$} (h);
		\draw[line width=0.3mm, bend left] (h) edge node {} (d);
		\draw[line width=0.3mm, above, shorten >=-1pt] (bb) edge node {$y^-$} (b);
		\draw[line width=0.3mm, bend right] (b) edge node {} (d);
		\draw[line width=0.3mm] (d) edge node {} (dd);
		\draw[line width=0.3mm, above] (dd) edge node {$x^+ = y^+$} (dm);
		\end{tikzpicture}

 \caption{Two right asymptotic sequences $x,y$.}
 \label{figureAsymptoticFibo}
\end{figure}
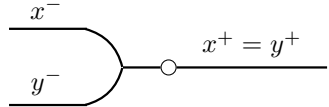

The classes of the asymptotic equivalence not reduced to one orbit are called \emph{right asymptotic classes} (they are called in~\cite{DonosoDurandMaassPetite2016} \emph{asymptotic components}).

\begin{example}
The Fibonacci shift $X$ has one right asymptotic class formed of the shifts of the two elements $x,y \in X$ such that $x^+ = y^+ = \varphi^\omega(a)$ where $\varphi^\omega(a)$ is the Fibonacci word, that is the right infinite word having all $\varphi^n(a)$ for $n \ge 1$ as prefixes.
Indeed, let $x, y \in X$ be such that $x^+ = y^+$ with $x \ne y$.
Then all finite prefixes of $x^+ = y^+$ are left-special and thus are prefixes of $\varphi^\omega(a)$ (see, for instance,~\cite{PytheasFogg2002}).
Thus $x^+ = y^+ = \varphi^\omega(a)$.
\end{example}

If $C$ is an asymptotic class, it is, by the definition of asymptotic equivalence, a union of orbits. 
The following result is proved in~\cite[Lemma 3.2]{DonosoDurandMaassPetite2016} under a weaker hypothesis that we shall not need here.
We give a proof for the sake of completeness.

\begin{proposition}
\label{propositionFiniteNbClasses}
Let $X$ be a shift space such that the sequence $s_n(X)$ is bounded by $k$.
Then the number of asymptotic classes is finite and at most equal to $k$.
\end{proposition}
\begin{proof}
Let $(x_1,y_1), \ldots, (x_\ell,y_\ell)$ be $\ell$ pairs of distinct elements of $X$ belonging to asymptotic classes $C_1, \ldots, C_\ell$
such that for all 
$1 \le i \le \ell$, one has $x_i^+ = y_i^+$  and $(x_i)_{-1} \ne (y_i)_{-1}$.
For $n$ large enough the prefixes of length $n$ of the $x_i^+$ are $\ell$ distinct left-special words and thus $\ell \le s_n(X)$ since by Proposition~\ref{propositionCassaigne} the number of left-special words is bounded by $s_n(X)$.
This shows that the number of asymptotic classes is finite and bounded by $k$.
\end{proof}

Let $X$ be a shift space.
For an asymptotic class $C$ of $X$, we denote $\omega(C) = \Card(o(C))-1$ where $o(C)$ is the set of orbits contained in $C$.
For a right infinite word $u \in X^+$, let 
\begin{displaymath}
\ell_C(u) = \Card\{ a \in A \mid x^+ = au \mbox{ for some $x\in C$} \}.
\end{displaymath}
We denote by $LS_\omega(C)$ the set of right infinite words $u$ such that $\ell_C(u) \ge 2$.

The following statement can be seen as an infinite counterpart of Proposition~\ref{propositionCassaigne}.

\begin{proposition}\label{pro:eqomegaC}
Let $X$ be a  shift space and let $C$ be a right asymptotic class.
Then
\begin{equation}
\omega(C) = \sum_{u\in LS_\omega(C)}(\ell_C(u)-1)
\label{equationC}
\end{equation}
where both sides are simultaneously finite.
\end{proposition}

In order to prove Proposition~\ref{pro:eqomegaC}, we use the notion of a \emph{cluster of trees} that we now define. 

A \emph{cluster of trees} is an oriented directed graph which is the union of a (non-trivial) cycle $\Gamma$ and a family of disjoint trees (oriented from child to father) $T_v$ with root $v$ indexed by the vertices $v$ on $\Gamma$ (see Figure~\ref{figureClusterTrees}).
It is easy to verify that a finite connected graph is a cluster of trees if and only if every vertex has outdegree $1$ and there is a unique strongly connected component.

\begin{figure}[hbt]
	\centering
	\tikzset{node/.style={circle,draw,minimum size=0.2cm,inner sep=0pt}}
	\begin{tikzpicture}
		\node[node](n) {};
		\node[node](e) [below right = 0.7cm and 0.7 of n] {};
		\node[node](s) [below left = 0.7cm and 0.7 of e] {};
		\node[node](w) [above left = 0.7cm and 0.7cm of s] {};
		\node[node](n1) [above = 0.5cm of n] {};
		\node[node](n2) [above left = 0.5cm and 0.5cm of n1] {};
		\node[node](n3) [above right = 0.5cm and 0.5cm of n1] {};
		\node[node](w1) [left = 0.5cm of w] {};
		\node[node](w2) [above left = 0.5cm and 0.5cm of w1] {};
		\node[node](w3) [below left = 0.5cm and 0.5cm of w1] {};
		\node[node](e1) [above right = 0.5cm and 0.5cm of e] {};
		\node[node](e2) [below right = 0.5cm and 0.5cm of e] {};
		\draw[line width=0.3mm, bend left, ->] (n) edge node{} (e);
		\draw[line width=0.3mm, bend left, ->] (e) edge node{} (s);
		\draw[line width=0.3mm, bend left, ->] (s) edge node{} (w);
		\draw[line width=0.3mm, bend left, ->] (w) edge node{} (n);
		\draw[line width=0.3mm, ->] (n2) edge node{} (n1);
		\draw[line width=0.3mm, ->] (n3) edge node{} (n1);
		\draw[line width=0.3mm, ->] (n1) edge node{} (n);	
		\draw[line width=0.3mm, ->] (w2) edge node{} (w1);
		\draw[line width=0.3mm, ->] (w3) edge node{} (w1);
		\draw[line width=0.3mm, ->] (w1) edge node{} (w);	
		\draw[line width=0.3mm, ->] (e1) edge node{} (e);
		\draw[line width=0.3mm, ->] (e2) edge node{} (e);
	\end{tikzpicture}

	\caption{A cluster of trees.}\label{figureClusterTrees}
\end{figure}
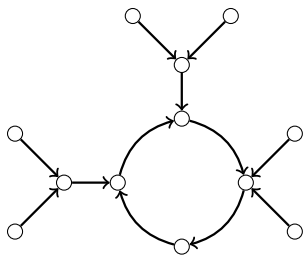

In a cluster of trees, the number of leaves (that is, the leaves of the trees $T_v$ not reduced to their root) is equal to $\sum_u ( d^-(u) - 1 )$, where $d^-$ stands for the indegree function and the sum runs over the set of internal nodes.
Indeed, this is true for one cycle alone since there are no leaves and every internal node has indegree 1.
The formula remains valid when suppressing a leaf in one of the trees not reduced to its root.

\begin{proofof}{ of Proposition~\ref{pro:eqomegaC}}
We assume that $\omega(C) \ge 1$ and we first suppose that $C$ does not contain periodic points which implies that $LS_\omega(C)$ does not contain periodic points either.
	
It is easy to verify that if $u,v \in LS_\omega(C)$, there exist $n,m \ge 0$ such that $\sigma^n(u) = \sigma^m(v)$.

We build a graph $T(C)$  as follows.
The set of vertices of $T(C)$ is $o(C) \cup LS_\omega(C)$.
There will be for each vertex $u$ of $T(C)$ at most one edge going out of $u$, called its father.

Let first $x\in C$ and let $u$ be the orbit of $x$.
There is, up to a shift of $x$, at least one $y \in C$ with $x \ne y$ such that $y^+ = x^+$.
Let $n \ge 0$ be the minimal integer such that $x_{-n} \ne y_{-n}$.
Then $v=\sigma^{-n+1}(x)^+$ is in $LS_\omega(C)$ and depends only on the orbit $u$ 
of $x$.
We choose the vertex $v$ as the father of $u$.

Next, for every $u\in LS_\omega(C)$, we consider  the minimal integer, if it exists, such that $v=\sigma^n(u)$ is in  $LS_\omega(C)$. 
Then we choose $v$ as the father of $u$.
	
Assume now that $\omega(C)$ is finite.
Then $LS_\omega(C)$ is also finite and $T(C)$ is a finite tree.
Indeed, if $u\in LS_\omega(C)$, there is at least one $x\in C$ such that $x^+=u$ and thus such
that $u$ is an ancester of the orbit of $x$.
By the claim made above, any two elements of $LS_\omega(C)$ have a common ancester.
Since $C$ does not contain periodic points, two vertices cannot be ancestors of one another. Thus there is a unique element of $LS_\omega(C)$ which has no father, namely the unique $u \in SL(C)$ with a maximal number of elements of $o(C)$ as descendants.
Since it is an ancestor of all vertices of $T(C)$, this shows that $T(C)$ is a finite tree.

Formula~\eqref{equationC} now follows from the fact that in any finite tree with $n$ leaves and and a set $V$ of internal vertices, one has $n-1 = \sum_{v\in V}( d^{-}(v)-1 )$.
	
Assume next that the right hand side of Equation~\eqref{equationC} is finite.
Then the set $LS_\omega(C)$ is finite and thus $T(C)$ is again a tree with a finite number of internal nodes.
Since the degree of each node is finite, it implies that it has also a finite number of leaves.
Thus $\omega(C)$ is finite and Equation~\eqref{equationC} also holds.
	
Finally, assume that $C$ contains a periodic point.
It follows from the definition of an asymptotic class that there is	exactly one such periodic orbit, since two periodic points having a common tail are in the same orbit. 
	
The proof follows the same lines as in the first case, but this time $T(C)$ will be a cluster of trees instead of a tree.

The set of leaves of $T(C)$ is, as above, the set $o(C)$ of non periodic orbits and the	the other vertices are the elements of $LS_\omega(C)$.
The unique father of a vertex is defined in the same way as above.
The fact that there is a unique strongly connected component is a consequence of the fact that there is a unique periodic orbit in $C$.
Finally, Formula~\eqref{equationC} holds with since the number of leaves is equal to $\sum ( d^{-}(u) - 1 ) -1$, where the sum runs over the set of internal nodes and the $-1$ corresponds to the unique periodic orbit.
\end{proofof}

\begin{example}
Consider again the image $\alpha(X)$ of the Tribonacci shift by the morphism $\alpha : a \mapsto a, b \mapsto a, c \mapsto c$ (Example~\ref{exampleImageTribo}).

\begin{figure}[hbt]
	\centering
	\tikzset{title/.style={minimum size=-0cm,inner sep=0pt}}
	\tikzset{small/.style={circle,draw,minimum size=0.2cm,inner sep=0pt}}
	\tikzset{node/.style={rectangle,draw,rounded corners=1.4ex}}
	\begin{tikzpicture}
		\node[title](hhh) {$x$};
		\node[title](hh) [right= 1.4cm of hhh] {};
		\node[title](bbh) [below= 0.4cm of hhh] {$y$};
		\node[title](bbhm) [right= 0.2cm of bbh] {};
		\node[title](bh) [right= 0.2cm of bbh] {};
		\node[title](bbb) [below= 1cm of bbh] {$z$};
		\node[title](bbbm) [right= 0.2cm of bbb] {};
		\node[title](bb) [above right= 0.4cm and 1cm of bbb] {};
		\node[small](bm) [right= 0.2cm of bb] {};
		\node[title](h) [above right= 0.6cm and 1cm of bb] {};
		\node[small](hm) [right= 0.2cm of h] {};
		\node[title](hf) [right= 1.8cm of hm] {};
		\draw[line width=0.3mm, shorten >=-1pt] (hhh) edge node {} (hh);
		\draw[line width=0.3mm, bend left] (hh) edge node {} (h);
		\draw[line width=0.3mm, shorten >=-1pt] (bbh) edge node {} (bbhm);
		\draw[line width=0.3mm, bend left] (bbhm) edge node {} (bb);
		\draw[line width=0.3mm, shorten >=-1pt] (bbb) edge node {} (bbbm);
		\draw[line width=0.3mm, bend right] (bbbm) edge node {} (bb);
		\draw[line width=0.3mm] (bb) edge node {} (bm);
		\draw[line width=0.3mm, bend right] (bm) edge node {} (h);
		\draw[line width=0.3mm] (h) edge node {} (hm);
		\draw[line width=0.3mm] (hm) edge node {} (hf);

		\node[node](y) [right = 2.8cm of hm] {$y$};
		\node[node](z) [below = 0.4cm of y] {$z$};
		\node[node](aax+) [above right = 0cm and 1cm of z] {$aax^+$};
		\node[node](x) [above = 0.6cm of aax+] {$x$};
		\node[node](x+) [above right= 0.2cm and 1cm of aax+] {$x^+$};
		\path[draw,thick, shorten >=-1pt]
			(y) edge node {} (aax+)
			(z) edge node {} (aax+)
			(x) edge node {} (x+);
		\path[draw,thick, shorten <=-2pt, shorten >=-1pt]
			(aax+) edge node {} (x+);
	\end{tikzpicture}

 \caption{The asymptotic class $C$ and the tree $T(C)$.}
 \label{figureAsymptoticClassesImageTribo}
\end{figure}
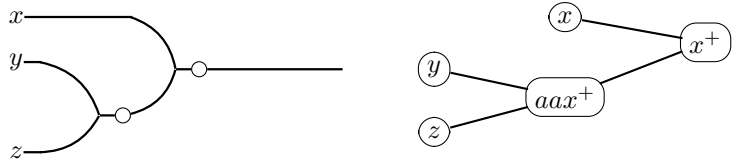

There is one asymptotic class $C$ made of three orbits represented in Figure~\ref{figureAsymptoticClassesImageTribo} on the left.
The class is formed of the orbits of $x,y,z$ where $x^+ = \alpha(\varphi^\omega(a))$ and $y^+ = z^+ = aax^+$.
The tree $T(C)$ is shown on the right.
\end{example}

In the next example we use the notation $u^\omega$ for the right infinite word $uuu\cdots$ and symmetrically $^\omega u$ for the left-infinite word $\cdots uuu$. 

\begin{example}
Let $X$ be the shift space which is the closure under the shift of the set
$\{^\omega c.(ab)^\omega \; \cup \; \,^\omega d.(ab)^\omega \; \cup \; ^\omega (ab)\cdot(ab)^\omega\}$.
The shift has just one right asymptotic class $C$, the one associated to the tail $(ab)^\omega$, containing three orbits.
Since the tail $(ab)^\omega$ can be prolonged on the left by either $c,d$ or $b$, we have that Formula~\ref{equationC} is verified. 
The cluster of trees $T(C)$ is represented in Figure~\ref{figureCluster}
where we denote by $xy$ the orbit of $x\cdot y$.

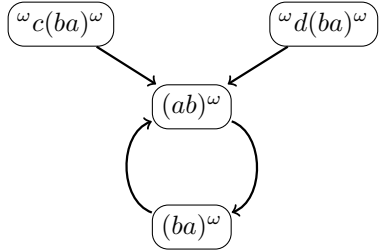
\begin{figure}[hbt]
		\centering
		\tikzset{node/.style={rectangle,draw,rounded corners=1.4ex}}
		\begin{tikzpicture}
			\node[node](ab) {$(ab)^\omega$};
			\node[node](ba) [below = 1cm of ab] {$(ba)^\omega$};
			\node[node](cab) [above left = 0.5cm and 0.5 of ab] {$^\omega c(ba)^\omega$};
			\node[node](dab) [above right = 0.5cm and 0.5cm of ab] {$^\omega d(ba)^\omega$};
			\draw[line width=0.3mm, bend left=70, ->] (ab) edge node{} (ba);
			\draw[line width=0.3mm, bend left=70, ->] (ba) edge node{} (ab);
			\draw[line width=0.3mm, ->] (cab) edge node{} (ab);
			\draw[line width=0.3mm, ->] (dab) edge node{} (ab);
		\end{tikzpicture}
		\caption{The cluster $T(C)$.}\label{figureCluster}
	\end{figure}
\end{example}

Let us now deduce from Proposition~\ref{propositionSG} a characterization of eventually dendric shift spaces in terms of asymptotic classes.
For a shift space $X$, denote
\begin{displaymath}
\omega(X) = \sum\omega(C)
\end{displaymath}
where the sum is over the asymptotic classes $C$ of $X$.

\begin{proposition}
\label{propositionCaracterizationDendric}
A  shift space $X$ is eventually dendric if and only if:
\begin{enumerate}
\item The sequence $s_n(X)$ is eventually constant, and
\item We have $\lim s_n(X)=\omega(X)$.
\end{enumerate}
\end{proposition}
\begin{proof}
Assume first that $X$ is eventually dendric. 
Then assertion 1 holds by Proposition~\ref{propositionComplexity}.
To prove assertion 2, consider an integer $n$ large enough so that the condition of Proposition~\ref{propositionSG} holds (it implies that $s_m(X)$ is constant for $m \ge n$).
Let us consider an asymptotic class $C$.  

Let $\pi$ be the map assigning to $u \in A^\N$  its prefix of length $n$.
Then $\pi$ maps $LS_\omega(C)$ into $LS_n(X)$.
The map $\pi$ is injective since otherwise some word in $LS_\ge n(X)$ would have more than one extension on the right, contrary to Proposition~\ref{propositionSG}.
Next the sets $\pi(LS_\omega(C))$ for all asymptotic classes $C$ form a partition of $LS_n(X)$. 

Thus, by Equation~\eqref{eqs_n},
\begin{eqnarray*}
s_n(X) & = & \sum_{w \in LS_n(X)}(\ell(w)-1) = \sum_{C} \sum_{u\in LS_\omega(C)}(\ell_C(u)-1) \\
 & = &  \sum_C\omega(C)
\end{eqnarray*}
where the last equality follows from Equation~\eqref{equationC}.

Conversely, if the two conditions are satisfied, let $n$ be large enough so that $s_m(X) = s_n(X)$ for all $m \ge n$.
We may also assume that $n$ is large enough so that the prefixes of length $n$ of the words of $LS_\omega(C)$ for every asymptotic class $C$ are distinct.
Then, every word $w$ of $LS_n(X)$ has exactly one right extension $wb$ in $LS_{n+1}(X)$.
It is moreover such that $\ell(w) = \ell(wb)$ and thus $X$ is eventually dendric by Proposition~\ref{propositionSG}.
\end{proof}

\section{Simple trees}
\label{sectionSimpletrees}
The \emph{diameter} of a tree is the maximal length of simple paths.
A tree is \emph{simple} if its diameter is at most $3$.
Note that if the simple tree is the extension graph $\E_n(w)$ in some shift space $X$ of a bispecial word $w$, then the diameter of $\E_n(w)$ is equal to $3$ and this happens if and only if any two vertices of $\E_n(w)$ on the same side (that is, both in $L_n(x)$ or both in $R_n(w)$) are connected to a common vertex on the opposite side.

For example, if $X$ is the Fibonacci shift, then $\E_1(a)$ is simple while $\E_3(a)$ is not (see Example~\ref{exampleFibo}).

We prove the following additional property of the graphs $\E_k(w)$.

\begin{proposition}
\label{propositionExt}
Let $X$ be an eventually dendric shift space.
For any $k \ge 1$ there exists an $n \ge 1$ such that $\E_k(w)$ is a simple tree for every $w \in \LL_{\ge n}(X)$.
\end{proposition}

We first prove the following lemma.

\begin{lemma}
\label{lemmapw}
Let $X$ be an eventually dendric shift space.
For every $k \ge 1$ there is an $n \ge 1$ such that if $p,w \in \LL(X)$ with $|p| \le k$ and $|w| \ge n$ are such that $pw,w \in LS(X)$, then $pw$, $w$ have a unique right extension in $LS(X)$ for some letter $b\in A$ which is moreover such that $\ell(pwb) = \ell(pw)$ and $\ell(wb) = \ell(w)$.
\end{lemma}
\begin{proof}
Consider two asymptotic classes $C,D$ and let $u \in LS_\omega(C)$, $v \in LS_\omega(D)$.
If $C,D$ are distinct, we cannot have $pu = v$ for some word $p$.
Thus there is an integer $n$ such that if $w$ is the prefix of length $n$ of $u$, then $pw$ is not a prefix of $v$.
Since there is a finite number of words $p$ of length at most $k$, a finite number of asymptotic classes (by Proposition~\ref{propositionFiniteNbClasses}) and since for each such class the set $LS_\omega(C)$ is finite, we infer that for every $k$ there exists an $n$ such that for every pair of asymptotic classes $C,D$ and any $u \in S(C),v\in LS(D)$, if $w$ is a prefix of $u$ and $pw$ a prefix of $v$, with $|p| \le k$ and $|w| = n$, then $C = D$.

Next, assume that $w$ is a prefix of $u$ and $pw$ a prefix of $v$ with $u,v \in S(C)$ for some asymptotic class $C$.
If $v \ne pu$, then there is a right extension $w'$ of $w$ such that $pw'$ is not a prefix of $v$.
By contraposition, if $n$ is large enough, we have $v = pu$. 

We thus choose $n$ large enough so that:
\begin{enumerate}
\item All elements of $LS_\omega(C)$ for all asymptotic components $C$ have distinct prefixes of length $n$;
\item For every pair of asymptotic classes $C,D$ and any $u \in LS_\omega(C),v\in LS(D)$, if $w$ is prefix of $u$ and $pw$ is prefix of $v$ with $|p| \le k$ and $|w| = n$ then $C = D$ and $pu = v$.
\end{enumerate}
We moreover assume that $n$ is large enough so that the condition of Proposition~\ref{propositionSG} holds.

Consider $p,w$ with $|p|=k$ and $|w|=n$ such that $pw,w$ are left-special.
By condition 1, there are asymptotic components $C,D$ and elements $u \in LS_\omega(C)$ and $v \in LS(D)$ such that $w$ is a prefix of $u$ and $pw$ a prefix of $v$.
Because of condition 2, we must have $\sigma^k(v)=u$ (and in particular $C=D$).
Thus there is a unique letter $b \in A$ such that $wb,pwb \in LS(X)$ which is moreover such that $\ell(wb)=\ell(w)$ and $\ell(pwb)=\ell(pw)$ by  Proposition~\ref{propositionSG}.
\end{proof}

\begin{proofof}{of Proposition~\ref{propositionExt}}
We choose $n$ such that Proposition~\ref{propositionSG} and Lemma~\ref{lemmapw} hold.

We prove by induction on $\ell$ with $1 \le \ell \le k$ that 
$\E_\ell(w)$ is a simple tree
and thus that for any $p,q \in L_\ell(w)$ there is an $r \in R_k(w)$ such that $pwr,qwr \in \LL(X)$.

The property is true for $\ell=1$.
Indeed, set $p=a$ and $q=b$.
Apply iteratively Proposition~\ref{propositionSG} to obtain letters $c_1, \ldots, c_k$ such that $\ell(w c_1 \cdots c_i) = \ell(w c_1 \cdots c_ic_{i+1})$ and set $r = c_1 \cdots c_k$.
Then $awr,bwr \in \LL(X)$.

Assume next that the property is true for $\ell-1$ and consider $ap,bq \in L_\ell(w)$ with $a,b \in A$.
Replacing if necessary $w$ by some longer word, we may assume that $p,q$ end with different letters and thus that $w$ is left-special.
By the induction hypothesis, there is a word $r \in R_k(w)$ such that $pwr,qwr \in \LL(X)$.
By Lemma~\ref{lemmapw}, the first letter of $r$ is the unique letter $c$ such that $\ell(pwc) = \ell(pw)$ and $\ell(qwc) = \ell(qw)$.
Thus $apwc,bqwc \in \LL(X)$.
Applying Lemma~\ref{lemmapw} iteratively in this way, we obtain that $apwr,bqwr \in \LL(X)$.
\end{proofof}

\section{Conjugacy}
\label{sectionConjugacy}
Let $A,B$ be two alphabets, and $X \subset A^\Z$ and $Y \subset B^Z$ be two shift spaces.
A map $\phi: X \to Y$ is called a \emph{sliding block code} if there exists $m,n \in \N$ and a map $f : \LL_{m+n+1}(X)\to B$ such that $\phi(x)_i = f(x_{i-m} \cdots x_{i+n})$ for all $i \in \Z$ and $x=(x_i) \in X$.
It can be shown that a map $\phi : X \to Y$ is a sliding block code if and only if it is continuous and commutes with the shift, that is $\phi \circ \sigma_A = \sigma_B \circ \phi$ (see, for instance,~\cite{LindMarcus1995}).

Two shift spaces $X,Y$ are said to be \emph{conjugate} when there is a bijective sliding block code $\phi : X \to Y$.
The following result shows that the property of being eventually dendric is a dynamical property.

\begin{theorem}
\label{theoremConjugacy}
The class of eventually dendric shift spaces is closed under conjugacy.
\end{theorem}

We first treat the following particular case of conjugacy.
Let $X$ be a shift space on the alphabet $A$ and let $k \ge 1$.
Let $f : \LL_k(X) \rightarrow A_k$ be a bijection from the set $\LL_k(X)$ of blocks of length $k$ of $X$ onto an alphabet $A_k$.
The map $\gamma_k : X \rightarrow A_k^{\Z}$ defined for $x \in X$ by $y = \gamma_k(x)$ if for every $n \in \Z$
\begin{displaymath}
y_n = f(x_n \cdots x_{n+k-1})
\end{displaymath}
is the $k$-th \emph{higher block code} on $X$.
The shift space $X^{(k)} = \gamma_k(X)$ is called the $k$-th \emph{higher block shift space} of $X$.
It is well known that the $k$-th higher block code is a conjugacy.

We extend the bijection $f : \LL_k(X) \rightarrow A_k$ to a map still denoted $f$ from $\LL_{\ge k}(X)$ to $\LL_{\ge 1}(X^{(k)})$ by $f(a_1 a_2 \cdots a_n) = f(a_1 \cdots a_k) \cdots f(a_{n-k+1} \cdots a_n)$.
Note that all nonempty elements of $\LL(X^{(k)})$ are image by $f$ of an element of $\LL(X)$, that is, $\LL(X^{(k)}) = \{ f(w) \mid w \in \LL_{\ge k}(X) \} \cup \{ \varepsilon \}$.

\begin{example}
\label{ex:Fibok}
Let $X$ be the Fibonacci shift.
We show that the $2$-block extension $X^{(2)}$ of $X$ is eventually dendric with threshold $1$.
Set $A_2 = \{u,v,w\}$ with $f : aa \mapsto u, ab \mapsto v, ba \mapsto w$.
Since $X$ is dendric, the graph $\E_1(w)$ is a tree for every word $w \in \LL(X^{(2)})$ of length at least $1$ (but not for $w=\varepsilon$).
Thus $X^{(2)}$ is eventually dendric.
It is actually a tree shift space of characteristic $2$ since the graph $\E_1(\varepsilon)$ is the union of two trees (see Figure~\ref{figure2blocks}).

\begin{figure}[hbt]
	\centering
	\tikzset{node/.style={rectangle,draw,rounded corners=1.4ex}}
	\begin{tikzpicture}
		\node[node](eul) {$u$};
		\node[node](evl) [below = 0.3cm of eul] {$v$};
		\node[node](ewl) [below = 0.3cm of evl] {$w$};
		\node[node](eur) [right = 1.5cm of eul] {$u$};
		\node[node](evr) [below = 0.3cm of eur] {$v$};
		\node[node](ewr) [below = 0.3cm of evr] {$w$};
		\path[draw,thick, shorten <=-1pt, shorten >=-1pt]
			(eul) edge node {} (evr)
			(evl) edge node {} (ewr)
			(ewl) edge node {} (evr);
		\path[draw,thick, shorten <=-2pt, shorten >=-3pt]
			(ewl) edge node {} (eur);
		\node[node](vwul) [below right = 0cm and 1.5cm of eur] {$u$};
		\node[node](vwwl) [below = 0.5cm of vwul] {$w$};
		\node[node](vwur) [right = 1.5cm of vwul] {$u$};
		\node[node](vwvr) [below = 0.5cm of vwur] {$v$};
		\path[draw,thick]
			(vwul) edge node {} (vwur);
		\path[draw,thick, shorten <=-2pt, shorten >=-2pt]
			(vwul) edge node {} (vwvr)
			(vwwl) edge node {} (vwur);
	\end{tikzpicture}

 \caption{The extension graphs $\E_1(\varepsilon)$ and $\E_1(vw)$.}
 \label{figure2blocks}
\end{figure}
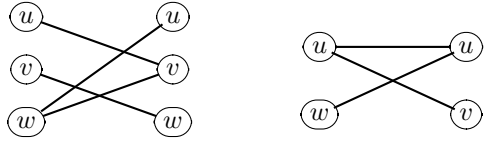
\end{example}

\begin{lemma}
	\label{lemmaHigher}
For every $k \ge 1$, the $k$-th higher block shift space $X^{(k)}$ is eventually dendric if and only if $X$ is eventually dendric.
\end{lemma}
\begin{proof}
We define for every $w \in \LL_{\ge k}(X)$ a map from $\E_1(w)$ to $\E_1(f(w))$ as follows.

To every $a \in L_1(w)$, we associate the first letter $\lambda(a)$ of $f(aw)$ and to every $b \in R_1(w)$, we associate the last letter $\rho(b)$ of $f(wb)$.
Then, since $f(awb) = \lambda(a) f(w) \rho(b)$, the pair $(a,b)$ is in $E_1(w)$ if and only if $(\lambda(a), \rho(b))$ is in $E_1(f(w))$.
Thus, the maps $\lambda,\rho$ define an isomorphism from $\E_1(w)$ onto $\E_1(f(w))$.

Thus we conclude that $X$ is eventually dendric with threshold $m$ if and only if $X^{(k)}$ is eventually dendric with threshold $M$ with $0 \le M \le \sup(1,m-k+1)$.
\end{proof}

\begin{example}
Let $X$ be the Fibonacci shift.
For all $k \ge 2$, $X^{(k)}$ is an eventually dendric shift space with threshold $1$.
\end{example}

\begin{example}
Let $X$ be the shift space associated to the two-sided infinite word $\cdots ab ab \cdot ab ab \cdots$.
$X$ is an eventually dendric shift space with threshold $1$ (the empty word has 2 connected components).
For every $k \ge 1$, the shift space $X^{(k)}$ is eventually dendric with threshold $1$.
\end{example}

A morphism $\alpha : A^* \rightarrow B^*$ is called \emph{alphabetic} if $\alpha(A) \subset B$.

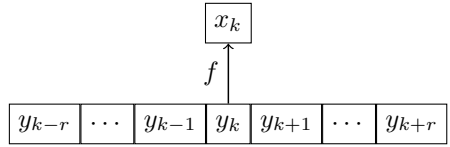
\begin{figure}[hbt]
	\centering
	\tikzset{node/.style={rectangle,draw,minimum size=15}}
	\begin{tikzpicture}
		\node[node](xk) {$x_k$};
		\node[node](yk) [below = 0.8cm of xk] {$y_k$};
		\node[node](yk+1) [right = 0cm of yk] {$y_{k+1}$};
		\node[node](yk+) [right = 0cm of yk+1] {$\cdots$};
		\node[node](yk+r) [right = 0cm of yk+] {$y_{k+r}$};
		\node[node](yk-1) [left = 0cm of yk] {$y_{k-1}$};
		\node[node](yk-) [left = 0cm of yk-1] {$\cdots$};
		\node[node](yk-r) [left = 0cm of yk-] {$y_{k-r}$};
		\draw[line width=0.2mm, ->, left] (yk) edge node {$f$} (xk);
	\end{tikzpicture}

 \caption{The sliding block code.}
 \label{figureSlidingBlock}
\end{figure}
\begin{lemma}
\label{lemmaMorphism}
Let $X$ be an eventually dendric shift space on the alphabet $A$ and let  $\alpha : A^* \rightarrow B^*$ be an alphabetic morphism which induces a conjugacy from $X$ onto a shift space space $Y$.
Then $Y$ is eventually dendric.
\end{lemma}
\begin{proof} 
Since $\alpha$ is invertible, there exists an integer $r \ge 0$ and a map $f : B^{2r+1} \rightarrow A$ such that for $x=(x_k)_{k\in \Z}$ and $y=(y_k)_{k\in\Z}$, one has $y = \alpha(x)$ if and only if for every $k \in \Z$, one has (see Figure~\ref{figureSlidingBlock})
\begin{displaymath}
x_k = f(y_{k-r} \cdots y_{k-1} y_k y_{k+1} \cdots y_{k+r}).
\end{displaymath}

We extend the definition of $f$ to a map from $\LL_{\ge 2r+1}(X)$ to $A$
as follows.
For $w = b_{1-r} \cdots b_{n+r}\in\LL_{\ge 2r+1}(Y)$,
set $f(w) = a_1 \cdots a_n$ where $a_i = f(b_{i-r} \cdots b_i \cdots b_{i+r})$.
Note  that if $u=f(w)$ and $w=svt$ with $s,t\in\LL_r(Y)$, then $v = \alpha(u)$ (see Figure~\ref{figurefuv}).

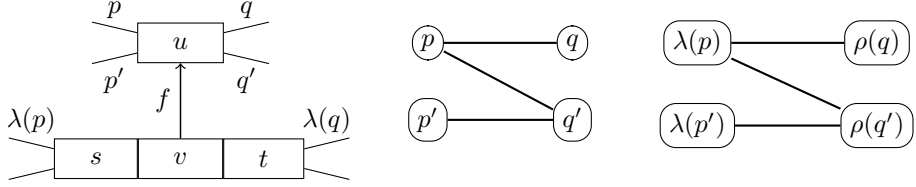
\begin{figure}[hbt]
\centering
 \tikzset{node/.style={rectangle,draw,minimum size=15}}
 \tikzset{title/.style={minimum size=-0cm,inner sep=0pt}}
 \tikzset{circ/.style={rectangle,draw,rounded corners=1.4ex}}
 \begin{tikzpicture}
  \node[node](u) {$\quad u \quad$};
  \node[node](v) [below = 1cm of u] {$\quad v \quad$};
  \node[node](t) [right = 0cm of v] {$\quad t \quad$};
  \node[node](s) [left = 0cm of v] {$\quad s \quad$};
  \draw[line width=0.2mm, ->, left] (v) edge node {$f$} (u);
  \node[title](p) [above left = 0cm and 0.6cm of u] {};
  \node[title](p2) [below left = 0cm and 0.6cm of u] {};
  \node[title](q) [above right = 0cm and 0.6cm of u] {};
  \node[title](q2) [below right = 0cm and 0.6cm of u] {};
  \node[title](lp) [above left = 0cm and 0.6cm of s] {};
  \node[title](lp2) [below left = 0cm and 0.6cm of s] {};
  \node[title](lq) [above right = 0cm and 0.6cm of t] {};
  \node[title](lq2) [below right = 0cm and 0.6cm of t] {};
  \path[-]
   (u) edge node[above] {$p$} (p)
   (u) edge node[below] {$p'$} (p2)
   (u) edge node[above] {$q$} (q)
   (u) edge node[below] {$q'$} (q2)
   (s) edge node[above] {$\lambda(p)$} (lp)
   (s) edge node {} (lp2)
   (t) edge node[above] {$\lambda(q)$} (lq)
   (t) edge node {} (lq2);
   
  \node[circ](kp) [right = 2.5cm of u] {$p$};
  \node[circ](kp2) [below= 0.5cm of kp] {$p'$};
  \node[circ](kq) [right= 1.5cm of kp] {$q$};
  \node[circ](kq2) [below= 0.5cm of kq] {$q'$};
  \path[draw,thick]
   (kp) edge node {} (kq)
   (kp) edge node {} (kq2)
   (kp2) edge node {} (kq2);
   
  \node[circ](1p) [right = 1cm of kq] {$\lambda(p)$};
  \node[circ](1p2) [below= 0.5cm of 1p] {$\lambda(p')$};
  \node[circ](1q) [right= 1.5cm of 1p] {$\rho(q)$};
  \node[circ](1q2) [below= 0.5cm of 1q] {$\rho(q')$};
  \path[draw,thick]
   (1p) edge node {} (1q)
   (1p) edge node {} (1q2)
   (1p2) edge node {} (1q2);
 \end{tikzpicture}
\caption{The map $f$ (on the left), the graph $\E_k(u)$ (on the center) and the graph $\E_1(w)$ (on the right).}
\label{figurefuv}
\end{figure}

Let $n$ be the integer given by Proposition~\ref{propositionExt} for $k=r+1$.
We claim that every graph $\E_1(w)$ for $|w| \ge n+2r$ is a tree.
Let indeed $s,t \in \LL_r(Y)$ and $v \in \LL_{\ge n}(Y)$ be such that $w=svt$.
Let $u = f(svt)$ (see Figure~\ref{figurefuv}).

Let
$
E'_k(u) = \{ (p,q) \in L_k(u) \times R_k(u) \mid \alpha(puq) \in BwB \}
$
and let $L'_k(u)$ (resp. $R'_k(u)$) be the set of $p \in L_k(u)$ (resp. $q \in R_k(u$)) which are connected to $L_k(u)$ (resp. $R_k(u)$) by an edge in $E'_k(u)$.
Let $\E'_k(u)$ be the subgraph of $\E_k(u)$ obtained by restriction to the set of vertices which is the disjoint union of $L'_k(u)$ and $R'_k(u)$ (and that has, thus, $E_k'(u)$ as set of edges.

Claim 1.
The graph $\E'_k(u)$ is a simple tree.
Indeed, by Proposition~\ref{propositionExt}, the graph $\E_k(u)$ is a simple tree.
We may assume that $u$ is $k$-bispecial (otherwise, the property is obviously true).
Let $(p,q)$ be an edge of $\E'_k(u)$.
Then $(p,q)$ is an edge of $\E_k(u)$ and since the latter is a simple tree either $p$ is the unique vertex in $L_k(u)$ such that $pu$ is right-special or $q$ is the unique vertex in $R_k(u)$ such that $uq$ is left-special (both cases can occur simultaneously).
Assume the first case, the other being proved in a symmetric way.
If $(p',q')$ is another edge of $\E'(u)$, then $(p,q')$ is an edge of $\E_k(u)$.
Since $\alpha(p) \in Bs$ and $\alpha(q) \in tB$, we have actually $(p,q') \in E'_k(u)$.
Thus $\E'_k(u)$ contains the two vertices of $\E_k(u)$ connected to more than one other vertex and this implies that $\E'_k(u)$ is a simple tree.

For $p \in L'_k(u)$, let $\lambda(p)$ be the first letter of $\alpha(p)$ and for $q \in R'_k(u)$, let $\rho(q)$ be the last letter of $\alpha(q)$.

Claim 2.
The graph $\E_1(w)$ is the image by the maps $\lambda, \rho$ of the graph $\E'_k(u)$.
Indeed, one has $(a,b)\in E_1(w)$ if and only if there exist $(p,q) \in E'_k(u)$ such that $\lambda(p)=a$ and $\rho(q)=b$.

Let us consider a graph homomorphism $\phi$ preserving bipartiteness and such such that left vertices are sent to left vertices and right vertices to right ones:
Then, it is easy to verify that the image of a simple tree by $\phi$ is again a simple tree.
Thus $\E_1(w)$ is a simple tree, which concludes the proof.
\end{proof}

We are now ready to prove the theorem.

\begin{proofof}{of Theorem~\ref{theoremConjugacy}}
Every conjugacy is a composition of a higher block code and an alphabetic morphism.
Thus Theorem~\ref{theoremConjugacy} is a direct consequence of Lemmas~\ref{lemmaHigher} and~\ref{lemmaMorphism}.
\end{proofof}

\begin{example}\label{exampleConjugacy}
We have seen in Example~\ref{exampleImageTribo} that the image of the Tribonacci shift by the morphism $\alpha : a \mapsto a, b \mapsto a, c \mapsto c$ is eventually dendric.
This is actually a consequence of Theorem~\ref{theoremConjugacy} since $\alpha$ is a conjugacy, as we have seen in Example~\ref{exampleImageTribo}.
The images of a Sturmian shift space by a non trivial alphabetic morphism have been investigated in~\cite{Vesely2018}.
\end{example}

\section{Minimal eventually dendric shifts}
\label{sectionProperty}
A shift space $X$ is \emph{irreducible} if for any $u,v \in \LL(X)$ there is a word $w$ such that $uwv \in \LL(X)$ (equivalently $\LL(X)$ is called \emph{recurrent}).

A nonempty shift space is \emph{minimal} if it does not contain properly another nonempty shift space.
As well known, $X$ is minimal if and only if it is \emph{uniformly recurrent}, that is for any $w \in \LL(X)$ there exists an $n \ge 0$ such that $w$ is a factor of any word in $\LL_n(X)$.
If $X$ is minimal and infinite, then there exists for every $w \in \LL(X)$ an integer $n \ge 1$ such that $w^n \notin \LL(X)$.
Indeed, otherwise, $\LL(X)$ contains the periodic word with period $w$ and thus $X$ is equal to the finite shift space formed by the shifts of $\cdots ww \cdot ww \cdots$.

A minimal shift space is irreducible but the converse is false, since for example the full shift $A^\Z$ is irreducible but not minimal as soon as $A$ has at least two elements.

Let $X$ be a shift space.
The set of \emph{complete return words} to a word $w \in \LL(X)$ is the set $\CR_X(w)$ of words having exactly two factors equal to $w$, one as a proper prefix and the other one as a proper suffix.
It is clear that $X$ is minimal if and only if it is irreducible and if for every word $w$ the set of complete return words to $w$ is finite.

If $wu$ is a complete return word to $w$, then $u$ is calleda (right) \emph{return word} to $w$.
We denote by $\RR_X(w)$ the set of return words to $w$.
Clearly $\Card(\CR_X(w))=\Card(\RR_X(w))$.

\begin{example}
	\label{exampleReturn}
Let $X$ be the Tribonacci shift  (see Example~\ref{exampleImageTribo}).
The image of $X$ under the morphism $\alpha: a,b \to a, c \to c$.
Then $\RR_X(a) = \{ a,ba,ca \}$ and $\RR_X(c) = \{ abac,ababac,abaabac \}$.
\end{example}

By a result of~\cite{BalkovaPelantovaSteiner2008}, if $X$ is minimal and neutral (a fortiori if $X$ is dendric) the set $\RR(w)$ has $\Card(A)$ elements for every $w \in \LL(X)$.
This is not true anymore for eventually dendric shift spaces, as shown in the following example.

\begin{example}
Let $X$ be the Tribonacci shift and let $Y=\alpha(X)$ be, as in Example~\ref{exampleConjugacy} the image of $X$ under the morphism $\alpha : a,b \to a, c \to c$.
Then, using Example~\ref{exampleReturn}, we find $\RR_Y(a) = \{ a,ca \}$ while $\RR_Y(c) = \{ aaac,aaaaac,aaaaaac \}$.
\end{example}

We will prove that for eventually dendric sets, a weaker property is true.
It implies that the cardinality of sets of return words is eventually constant.

For $w \in \LL(X)$, set $\rho_X(w) = r_1(w)-1$ and for a set $W \subset \LL$, set $\rho_X(W) = \sum_{w \in W} \rho_X(w)$.
By the symmetric of  Proposition~\ref{propositionNeutral}, for every neutral word $w \in \LL(X)$, we have
\begin{equation}
\rho_X(w) = \sum_{a \in L_1(w)} \rho_X(aw).
\label{eqLeftAdditive}
\end{equation}

\begin{theorem}
	\label{theoremReturn}
Let $X$ be an irreducible shift space which is eventually dendric with threshold $m$.
For every $w \in \LL(X)$, the set $\RR_X(w)$ is finite.
Moreover, for every $w \in \LL_{\ge m}(X)$, we have
\begin{equation}
\Card (\RR_X(w) ) = 1 + \rho_X(\LL_m(X)).
\label{eqReturn}
\end{equation}
\end{theorem}
Note that for $m=0$, we obtain $\Card(\RR_X(w)) = \Card(A)$ since $\rho_X(\varepsilon) = \Card(A)-1$.

A \emph{prefix code} (resp. a \emph{suffix code}) is a set $X$ of words such that none of them is a prefix (resp. a suffix) of another one.

A prefix code (resp. a suffix code) $U \subset \LL(X)$ is called \emph{$X$-maximal} if it is not properly contained in a prefix code (resp. suffix code) $Y \subset \LL(X)$ (see, for instance,~\cite{BerstelDeFelicePerrinReutenauerRindone2012}).

\begin{proposition}
\label{propositionBound}
Let $X$ be a shift space which is eventually dendric with threshold $m$.
Then $\rho_X(U)$ is finite for every suffix code $U\subset \LL(X)$.
If $U$ is a finite $X$-maximal suffix code with $U \subset \LL_{\ge m}(X)$, then
\begin{equation}
\rho_X(U) \ = \rho_X(\LL_m(X)).
\label{equationAdditive2}
\end{equation}
\end{proposition}
\begin{proof}
For any suffix code $U\subset \LL(X)$, let $U_m$ be the union 
$$U_m = \left( U \cap \LL_{<m}(X) \right) \cup \left( \LL_m(X) \cap S \right),$$
where $S$ is the set of words which are suffixes of some words of $U$.
Note that $U_m$ is a finite suffix code.
It is equal to $\LL_m(X)$ is $U$ is $X$-maximal and contained in $\LL_{\ge m}(X)$.

Assume first that $U\subset\LL(X)$ is a finite suffix code.
We prove the  by induction on the sum $\ell(U)$ of the lengths of the words of $U$ that

\begin{equation}
\rho_X(U) \le \rho_X(U_m)
\mbox{ with equality if }
U
\mbox{ is $X$-maximal and }
U \subset\LL_{\ge m}(X).
\label{eqInter}
\end{equation}

If all words of $U$ are of length at most $m$, then $U\subset U_m$ with equality if $U$ is $X$-maximal and $U\subset \LL_{\ge m}(X)$, since in this case $U_m = U = \LL_m(X)$.
Thus Equation~\eqref{eqInter} holds.
Otherwise, let $u \in U$ be of maximal length.
Set $u=av$ with $a \in A$.
Then $Av \cap \LL(X)\subset U$.
Set $U' = (U \setminus Av) \cup \{ v \}$.
Thus $U'$ is a suffix code with $\ell(U') < \ell(U)$ which is $X$-maximal if $U$ is $X$-maximal.
We have the inclusion $U\subset (U' \setminus v) \cup L_1(v)v$ with equality if $U$ is $X$-maximal.
Since $v$ is neutral, we have, by Equation~\eqref{eqLeftAdditive},
$$
\rho_X(U)  \le  \rho_X(U') -\rho_X(v) + \sum_{a \in L_1(v)} \rho_X(av) = \rho_X(U')
$$
with equality if $U$ is $X$-maximal.
By induction hypothesis, Equation~\eqref{eqInter} holds for $U'$.
Thus $\rho_X(U) \le \rho_X(U_m)$.
If $U$ is $X$-maximal and $U \subset \LL_{\ge m}(X)$, then $\rho_X(U) = \rho_X(\LL_m(X))$ since $U'_m = U_m = \LL_m(X)$.
Thus Equation~\eqref{eqInter} is proved.

If $U$ is infinite, then $\rho_X(U)$ is the supremum of the values of $\rho_X(V)$ on the finite subsets $V$ of $U$ and thus it is bounded by Equation~\eqref{eqInter}.
\end{proof}

\begin{proofof}{of Theorem~\ref{theoremReturn}}
Consider a word $w \in \LL(X)$ and let $P$ be the set of proper prefixes of $\CR(w)$.
For $p \in P$, denote $\alpha(p) = \Card\{ a \in A \mid pa \in P \cup \CR(w)\}-1$.
Then $\CR(w)$ is finite if and only if $P$ is finite.
Moreover in this case, since $\CR(w)$ is a prefix code, we have by a well known property of trees
\begin{equation}
\Card(\CR(w)) = \alpha(P)+1,
\label{equationCard}
\end{equation}
where $\alpha(P) = \sum_{p \in P} \alpha(p)$.

Let $U$ be the set of words in $P$ which are not proper prefixes of $w$.
We claim that $U$ is an $X$-maximal suffix code.

Indeed, if $u,vu \in U$, then $w$ is a proper prefix of $u$ and thus is an internal factor of $vu$, a contradiction unless $v=\varepsilon$.
Thus $U$ is suffix.

Consider $r \in \LL(X)$.
Then, since $\LL(X)$ is recurrent, there is some $s \in \LL(X)$ such that $wsr \in \LL(X)$. 
Let $u$ be the shortest prefix of $wsr$ which has a proper suffix equal to $w$ .
Then $u \in U$.
This shows that $U$ is an $X$-maximal suffix code.

We have $\alpha(p)=0$ for any proper prefix $p$ of $w$ since any word in $\CR(w)$ has $w$ as a proper prefix.
Next we have $\alpha(p) = \rho_X(p)$ for any $p \in U$.
Indeed, if $ua \in \LL(X)$ for $u \in P$ and $a \in A$, then $ua \in \CR(w) \cup P$ since $\LL(X)$ is recurrent.
Thus we have $\alpha(P) = \rho_X(U)$.

By Proposition~\ref{propositionBound}, $\rho_X(U)$ is finite.
Therefore, Equation~\ref{equationCard} shows that $\Card(\CR(w)) = \Card(\RR(w))$ is finite.

Assume finally that $|w| \ge m$.
Then $U \subset \LL_{\ge m}(X)$ and thus, by Proposition~\ref{propositionBound}, we have $\rho_X(U)=\rho_X(\LL_{m}(X))$.
Thus we have
\begin{displaymath}
\alpha(P) = \rho_X(\LL_{m}(X)).
\end{displaymath}
By Equation~\eqref{equationCard}, this implies 
Equation~\eqref{eqReturn}.
\end{proofof}

It is known that for dendric shift spaces,  irreducibility is enough to guarantee minimality~\cite{DolcePerrin2017}.
We obtain as a direct corollary of Theorem~\ref{theoremReturn} that this still holds for eventually dendric shifts.

\begin{corollary}
\label{theoremMinimalDendric}
An eventually dendric shift space is minimal if and only if it is irreducible.
\end{corollary}
\begin{proof}
Let $X$ be an irreducible shift space.
By Theorem~\ref{theoremReturn}, the set $\RR(w)$ is finite for every $w \in \LL(X)$.
Thus $X$ is minimal.
\end{proof}

Note that the proof shows that Theorem~\ref{theoremMinimalDendric} holds for the more general class of shift spaces which are \emph{eventually neutral}, in the sense that there is an integer $m$ such that every word of length at least $m$ is neutral.
This class includes the shift spaces $X$ such that $\LL(X)$ is neutral with characteristic $c$ introduced in \cite{DolcePerrin2017} and for which Theorem~\ref{theoremMinimalDendric} is proved in~\cite{DolcePerrin2017} with a similar proof.

Note also that Theorem~\ref{theoremMinimalDendric} shows that in a minimal eventually dendric shift space the cardinality of complete return words is bounded.
There exist minimal shift spaces which do not have this property (see~\cite[Example 3.17]{DurandLeroyRichomme2013}).

\section{Generalized extension graphs}
\label{sectionGeneralized}

We will now see how the conditions on extension graphs can be generalized to graphs expressing the extension by words having different length.

We will need the following notions.
Let $X$ be a shift space on an alphabet $A$.

A set $U \subset \LL(X)$ is said to be \emph{right $X$-complete} (resp. \emph{left $X$-complete}) if any long enough word of $\LL(X)$ has a prefix (resp. suffix) in $U$.

It is not difficult to show that a prefix code (resp. a suffix code) $U \subset \LL(X)$ is $X$-maximal if and only if it is right $X$-complete (resp. left $X$-complete) (see~\cite[Propositions 3.3.1 and 3.3.2]{BerstelDeFelicePerrinReutenauerRindone2012}).

For $U,V \subset A^*$ and $w \in \LL(X)$, let
$$
L_U(w) = \{ u \in U \mid uw \in \LL(X) \} \quad \mbox{ and } \quad R_V(w) = \{ v \in V \mid wv \in \LL(X) \}.
$$

Let $U \subset A^*$ (resp. $V \subset A^*$) be an suffix code (resp. prefix code) and $w \in \LL(X)$ be such that $L_U(w)$ is an $X$-maximal suffix code (resp. $R_V(w)$ is an $X$-maximal prefix code).
The \emph{generalized extension graph} of $w$ relative to $U,V$ is the following undirected bipartite graph $\E_{U,V}(w)$.
The set of vertices is the disjoint union of $L_U(w)$ and $R_V(w)$.
The edges are the pairs $(u,v) \in L_U(w) \times R_V(w)$ such that $uwv \in \LL(X)$.
In particular $\E_n(w) = \E_{\LL_n(X), \LL_n(X)}(w)$.

\begin{proposition}
\label{propositionE_n}
For every $n \ge 1$ and $m \ge 0$, the graph $\E_n(w)$ is a tree for all $w \in \LL_{\ge m}(X)$ if and only if $\E_{n+1}(w)$ is a tree for all words $w \in \LL_{\ge m}(X)$.
\end{proposition}

The proof uses the following statement.
The only if part is~\cite[Lemmas 3.8 and 3.10]{BertheDeFeliceDolceLeroyPerrinReutenauerRindone2015}.

\begin{lemma}
\label{lemmaMonatsHefte}
Let $X$ be a shift space and let $w \in \LL(X)$.
Let $U\subset \LL(X)$ be a finite $X$-maximal suffix code and let $V\subset \LL(X)$ be finite $X$-maximal prefix code.
Let $\ell \in \LL(X)$ be such that $A \ell \cap \LL(X) \subset U$ and such that $\E_{A,V}(\ell w)$ is a tree.
Set $U' = ( U \setminus A\ell ) \cup \{ \ell \}$.
The graph $\E_{U',V}(w)$ is a tree if and only if the graph $\E_{U,V}(w)$ is a tree.
\end{lemma}
\begin{proof}
We need only to prove the if part.

First, note that the hypothesis that $\E_{A,V}(\ell w)$ is a tree guarantees that the left vertices $A\ell$ in $\E_{U,V}(w)$ are clusterized: for any pair of vertices $a\ell, b\ell$ there exists a unique reduced path from $a\ell$ to $b\ell$ in $\E_{U,V}(w)$ using as left vertices only elements of $A\ell$.
Indeed, such a path exists since the subgraph $\E_{A\ell,V}(w)$ of $\E_{U,V}(w)$ is isomorphic to $\E_{A,V}(\ell w)$ that is connected.
Since $\E_{U,V}(w)$ is a tree, this path is unique.

Let $v, v' \in R_V(w)$ be two distinct vertices and let $\pi$ be the unique reduced path from $v$ to $v'$ in $\E_{U,V}(w)$.
We show that we can find a unique reduce path $\pi'$ from $v$ to $v'$ in $\E_{U',V}(w)$.

If $\pi$ does not pass by $A\ell$, we can simply define $\pi'$ as a path passing by the same vertices than $\pi$.
Otherwise, we can decompose $\pi$ in a unique way as a concatenation of a path $\pi_1$ from $v$ to a vertex in $A\ell$ not passing by $A\ell$ before, followed by a path from $A\ell$ to $A\ell$ (using on the left only vertices from $A\ell$) and a path $\pi_2$ from $A\ell$ to $v'$ without passing in $A\ell$ again.
We consider in $\E_{U',V}(w)$ the unique path $\pi_1'$ from $v$ to $\ell$ obtained by replacing the last vertex of $\pi_1$ by $\ell$ and the unique reduced path $\pi_2'$ from $\ell$ to $v'$ obtained by replacing the first vertex of $\pi_2$ by $\ell$.
In this case we define $\pi'$ as the concatenation of $\pi_1'$ and $\pi_2'$.

The reduced path $\pi'$ is unique.
Indeed, let us suppose that we have a different path $\pi^*$ from $v$ to $v'$ in $\E_{U',V}(w)$.
If $\pi^*$ does not pass (on the left) by $\ell$ then we would find a path having the same vertices in $\E_{U,V}(w)$ which is impossible since the graph is acyclic.
Let us suppose that both $\pi'$ and $\pi^*$ passes by $\ell$.
Without loss of generality let us suppose that we have a cycle in $\E_{U',V}(w)$ passing by $\ell$ and $v$ (the case with $v'$ being symmetric).
Let us define by $\pi_0'$ and $\pi_0^*$ the two distinct subpaths of $\pi'$ and $\pi^*$ respectively going from $v$ to $\ell$.
Since $\LL(X)$ is biextendable, we can find $a\ell, b\ell \in U$, with $a,b \in A$ not necessarily distinct, and two reduced paths $\pi_1$ from $v$ to $a\ell$ and and $\pi_2$ from $v$ to $b\ell$ in $\E_{U,V}(w)$ obtained from $\pi_0'$ and $\pi_0^*$ by replacing the vertex $\ell$ by $a\ell$ and $b\ell$ respectively.
From the remark at the beginning of the proof we know that we can find a reduced path in $\E_{U,V}(w)$ from $a\ell$ to $b\ell$.
Thus we can find a nontrivial cycle in $\E_{U,V}(w)$, which contradicts the acyclicity of the graph.
\end{proof}

A symmetric statement holds for $r \in \LL(X)$ such that $rA \cap \LL(X) \subset V$ and $\E_{U,A}(wr)$ is a tree, with $V' = ( V \setminus rA ) \cup \{ r  \}$: the graph $\E_{U,V}(w)$ is a tree if and only if $\E_{U,V'}(w)$ is a tree. 

\begin{lemma}
\label{lemmaFrancesco}
Let $n\ge 1$, let $m\ge 0$ and let $V$ be a finite $X$-maximal prefix code.
If $\E_{\LL_n(X),V}(w)$ is a tree for every $w \in \LL_{\ge m}(X)$ then
for each word
$\ell\in \LL_{\ge n-1}(X)$, the graph $\E_{A,V}(\ell w)$ is a tree. 
\end{lemma}
\begin{proof}
The graph $\E_{A,V}(\ell w)$ 
is obtained from $\E_{\LL_n(X),V}(\ell w)$ by identifying the vertices of $L_n(\ell w)$ ending with
the same letter. Since $\E_{\LL_n(X),V}(\ell w)$ is connected, $\E_{A,V}(\ell v)$ is also connected. 

Set 
$\ell=\ell'\ell''$ with $|\ell'|=n-1$. The graph $\E_{A,V}(\ell w)$
is isomorphic to $\E_{A\ell',V}(\ell''w)$ which is
a subgraph of $\E_n(\ell'' w)$ and thus it is acyclic. 

Thus $\E_{A,V}(\ell w)$ is a tree.
\end{proof}
A symmetric statement holds for $n\ge 1$ and $U$ a finite $X$-maximal suffix code:
If $\E_{U,\LL_n(X)}(w)$ is a tree for every $w \in \LL_{\ge m}(X)$ if and only
if $\E_{U,A}(wr)$ is a tree for every $r\in \LL_{\ge n-1}(X)$ and $w \in \LL_{\ge m}(X)$.

\begin{proofof}{of Proposition~\ref{propositionE_n}}
We proceed in several steps.

\noindent{Step 1.}
Assume first that $\E_n(w)$ is tree for every word $w \in \LL_{\ge m}(X)$. 
We fix some $w \in \LL_{\ge m}(X)$.

\noindent{Step 1.1}
We claim that for any finite $X$-maximal suffix code $U$
formed of words of length $n$ or $n+1$,
the graph $\E_{U,\LL_n(X)}(w)$ is a tree
by induction on $\gamma_{n+1}(U)=\Card(L_U(w)\cap A^{n+1})$.

The property is true for $\gamma_{n+1}(U) = 0$, since then $\E_{U,\LL_n(X)}(w) = \E_n(w)$.
Assume now that $\gamma_{n+1}(U) > 0$.
Let $a\ell$ with $a \in A$ be a word of length $n+1$ in $L_U(w)$.
Since $U$ is an $X$-maximal suffix code with words of length $n$ or $n+1$, we have $A\ell \cap \LL(X)\subset U$.
Let us consider $U' = (U \setminus A\ell) \cup \{ \ell \}$.
Since $\gamma_{n+1}(U') < \gamma_{n+1}(U)$, by induction hypothesis the graph $\E_{U',\LL_n(X)}(w)$ is a tree. Moreover, by Lemma~\ref{lemmaFrancesco}, the graph $\E_{A,\LL_n(X)}(\ell w)$ is a tree.

Thus, by assertion 1 of Lemma~\ref{lemmaMonatsHefte}, the graph $\E_{U,\LL_n(X)}(w)$ is a tree.
This proves the claim.

\noindent{Step 1.2}
We now claim that for any finite $X$-maximal prefix code $V$ formed of words of length $n$ or $n+1$, the graph $\E_{\LL_{n+1}(X),V}(w)$ is a tree by induction on $\delta_{n+1}(V) = \Card(R_V(w) \cap A^{n+1})$.

The property is true for $\delta_{n+1}(V) = 0$, since the graph $\E_{\LL_{n+1}(X),V}(w) = \E_{\LL_{n+1}(X),\LL_n(X)}(w)$, is a tree by Step 1.1.
Assume now that $\delta_{n+1}(V) > 0$.
Let $ra$ with $a \in A$ be a word of length $n+1$ in $R_V(w)$.
Since $V$ is an $X$-maximal prefix code with words of length $n$ or $n+1$, we have $rA \cap \LL(X)\subset U$.
Let us consider $V' = (V \setminus rA) \cup \{ r \}$.
Since $\delta_{n+1}(V') < \delta_{n+1}(V)$, by induction hypothesis the graph $\E_{\LL_{n+1}(X),V'}(w)$ is a tree.
Moreover, by the symmetric version of Lemma~\ref{lemmaFrancesco}, the graph $\E_{\LL_{n+1}(X),A}(wr)$ is a tree.
This proves the claim.

Since $\E_{n+1}(w)=\E_{\LL_{n+1}(X),\LL_{n+1}(X)}(w)$, we conclude that $\E_{n+1}(w)$ is a tree.

\noindent{Step 2}
Assume now that $\E_{n+1}(w)$ is a tree for every $w \in \LL_{\ge m}(X)$.
Fix some $w \in \LL_{\ge m}(X)$.

\noindent{Step 2.1}
We first claim that $\E_{U,\LL_{n+1}(X)}$ is a tree for every $X$-maximal suffix code $U$ formed of words of length $n$ or $n+1$ by induction on $\gamma_n(U) = \Card( L_U(w) \cap A^{n})$.

The property is true if $\gamma_n(U)=0$, since then $\E_{U,\LL_{n+1}(X)}(w)=\E_{n+1}(w)$.

Assume next that $\gamma_n(U) > 0$.
Let $\ell \in L_U(w) \cap A^n$.
Set $W = (U \setminus \{ \ell \}) \cup A\ell$ or equivalently $U = (W \setminus A\ell)\cup \{ \ell \}$.
Then $\delta_n(W) < \delta_n(U)$ and consequently $\E_{W,\LL_{n+1}(X)}(w)$ is a tree by induction hypothesis.
On the other hand, by Lemma~\ref{lemmaFrancesco}, the graph $\E_{A,\LL_{n+1}(X)}(\ell w)$ is also a tree.
By Assertion 2 of Lemma~\ref{lemmaMonatsHefte}, the graph $\E_{U,\LL_{n+1}(X)}(w)$ is a tree and thus the claim is proved.

\noindent{Step 2.2}
We now claim that $\E_{\LL_n(X),V}(w)$ is a tree for every $X$-maximal prefix code $V$ formed of words of length $n$ or $n+1$ by induction on $\delta_n(V) = \Card(R_V(w) \cap A^{n})$.

The property is true if $\delta_n(V) = 0$ by Step 2.1.
Assume now that $\delta_n(V) > 0$.
Let $r \in R_V(w) \cap A^n$ and let $T = (V \setminus \{ r \}) \cup rA$ or equivalently $V = ( T \setminus rA) \cup \{ r \}$.
Then $\delta_n(T) < \delta_n(V)$ and thus $\E_{\LL_n(X),T}(w)$ is a tree by induction hypothesis.
On the other hand, by the symmetric version of Lemma~\ref{lemmaFrancesco}, the graph $\E_{\LL_n(X),A}(wr)$ is also a tree.
By Assertion 2 of Lemma~\ref{lemmaMonatsHefte}, the graph $\E_{\LL_n(X),T}(w)$ is a tree and thus the claim is proved.

Since $\E_n(w) = \E_{U,V}(w)$ for $U=V=\LL_n(X)$, it follows from the claim that $\E_n(w)$ is a tree.
\end{proofof}

The following result shows that in the definition of eventually dendric shift spaces, one can replace the graphs $\E_1(w)$ by $\E_n(w)$ with the same threshold.

\begin{theorem}
\label{theoremGeneralized}
Let $X$ be a shift space.
For every $m \ge 1$, the following conditions are equivalent.
\begin{enumerate}
 \item[\rm(i)] $X$  is eventually dendric with threshold $m$,
 \item[\rm(ii)] the graph $\E_n(w)$ is a tree for every $n\ge 1$ and every word $w \in \LL_{\ge m}(X)$,
 \item[\rm(iii)] there is an integer $n \ge 1$ such that $\E_n(w)$ is a tree for every word $w \in \LL_{\ge m}(X)$.
\end{enumerate}
\end{theorem}
\begin{proof}
(i) $\Rightarrow$ (ii).
It is proved by ascending induction on $n$ using iteratively Proposition~\ref{propositionE_n}.

(ii) $\Rightarrow$ (iii).
It is obvious.

(iii) $\Rightarrow$ (i).
It is proved by descending induction on $n$ using Proposition~\ref{propositionE_n}.
\end{proof}

\section{Complete bifix decoding}
\label{sectionBifixDecoding}

Let $X$ be a shift space on an alphabet $A$. 
A subset of $\LL(X)$ is \emph{two-sided $X$-complete} if it is both left and right $X$-complete. 

A \emph{bifix code} is both a prefix code and a suffix code.
A bifix code $U \subset \LL(X)$ is $X$-maximal if it is not properly contained in a bifix code $V \subset \LL(X)$.
If a bifix code $U \subset \LL(X)$ is right $X$-complete (resp. left $X$-complete), it is an $X$-maximal bifix code since it is already an $X$-maximal
prefix code (resp. suffix code).
It can be proved conversely that if $X$ is irreducible, a finite bifix code is $X$-maximal if and only if it is two-sided $X$-complete (see~\cite[Theorem 4.2.2]{BerstelDeFelicePerrinReutenauerRindone2012}).
This is not true in general, as shown by the following example.

\begin{example}
Let $X$ be the shift space such that $\LL(X) = a^*b^*$.
The set $U = \{ aa,b \}$ is an $X$-maximal bifix code.
Indeed, it is a bifix code and it is left $X$-complete as one may verify.
However it is not right $X$-complete since no word in $ab^*$ has a prefix in $U$.
\end{example}

Let $X$ be a shift space and let $U$ be a two-sided $X$-complete finite bifix code.
Let $\varphi: B \rightarrow U$ be a coding morphism for $U$, that is, a bijection from an alphabet $B$ onto $U$ extended to a morphism from $B^*$ into $A^*$.
Then $\varphi^{-1}(\LL(X))$ is factorial and, since $U$ is two-sided complete, it is extendable. 
Thus it is the language of a shift space called the \emph{complete bifix decoding} of $X$ with respect to $U$.

For example, for any $n\ge 1$, the set $\LL_n(X)$ is a two-sided complete bifix code and the corresponding complete bifix decoding is the decoding
of $X$ by non-overlapping $n$-blocks.
It can be identified with the dynamical system $(X,\sigma^n)$.

In~\cite[Theorem 3.13]{BertheDeFeliceDolceLeroyPerrinReutenauerRindone2015} it is proved that the maximal bifix decoding of an irreducible dendric shift space is a dendric shift space.
Actually, the hypothesis that $X$ is irreducible is only used to guarantee that the $X$-maximal bifix code used for the decoding is also an $X$-maximal prefix code and an $X$-maximal suffix code.
In the definitions used here of a maximal bifix decoding, we do not need this hypothesis.

\begin{theorem}
\label{theoremCompleteBifixDecoding}
Any complete bifix decoding of an eventually dendric shift space is an eventually dendric shift space having the same threshold.
\end{theorem}

Note that any $X$-maximal suffix code $U$ one has $\Card(U) \ge \Card(X \cap A)$.
Indeed, every $a \in A$ appears as a suffix of (at least) an element of $X$.

\begin{lemma}
\label{lemmaGeneralized}
Let $X$ be an eventually dendric shift space with threshold $n$. 
For any $w \in \LL_{\ge n}(X)$, any $X$-maximal suffix code $U$ and any $X$-maximal prefix code $V$, the graph $\E_{U,V}(w)$ is a tree.
\end{lemma}
\begin{proof}
We use an induction on the sum of the lengths of the words in $U,V$.
The property is true if the sum is equal to $2 \Card(X \cap A)$.
Indeed, for every $w \in \LL_{\ge n}(X)$ one has $U = L(w)$ and $V = R(w)$ and thus $\E_{U,V}(w) = \E_1(w)$ is a tree.
Otherwise, we may assume that $U$ contains words of length at least $2$ (the case with $V$ being symmetrical).
Let $u \in U$ be of maximal length.
Set $u = a\ell$ with $a \in A$.
Since $U$ is an $X$-maximal suffix code, we have $A\ell \cap \LL(X) \subset U$.
Set $U' = (U \setminus A\ell) \cup \{ \ell \}$.
By induction hypothesis, the graphs $\E_{U',V}(w)$ and $\E_{A,V}(\ell w)$ are trees.
Thus, by Lemma~\ref{lemmaMonatsHefte}, $\E_{U,V}(w)$ is also a tree.
\end{proof}

\begin{proofof}{of Theorem~\ref{theoremCompleteBifixDecoding}}
Assume that $X$ is eventually dendric with threshold $n$.
Let $\varphi: B \rightarrow U$ be a coding morphism for $U$ and let $Y$ be the decoding of $X$ corresponding to $U$.
Consider a word $w$ of $\LL(Y)$ of length at least $n$.
By Lemma~\ref{lemmaGeneralized}, and since $|\varphi(w)|\ge n$, the graph $\E_{U,U}(\varphi(w))$ is a tree.
But for $b,c \in B$, one has $bwc \in \LL(Y)$ if and only if $\varphi(bwc) \in \LL(X)$, that is, if and only if $(\varphi(b), \varphi(c)) \in E_1(\varphi(w))$.
Thus $\E_1(w)$ is isomorphic to $\E_{U,U}(\varphi(w))$ and thus $\E_1(w)$ is a tree.
This shows that $Y$ is eventually dendric with threshold $n$.
\end{proofof}

\begin{example}
\label{ex:Fibobifix}
Let $X$ be the Fibonacci shift. Then $U = \{ aa, aba, b \}$ 
is an $X$-maximal bifix code.
Let $\varphi : \{ u,v,w \} \to U$ be the coding morphism for $U$ 
defined by $\varphi: u \mapsto aa, v \mapsto aba, w \mapsto b$.
The  complete bifix decoding of $X$ with respect to $U$ is an eventually dendric shift space with threshold $0$.
It is actually the natural coding of an interval exchange transformation on three intervals (see~\cite{bifixcodesintervalexchanges}).
The extension graphs $\E_1(\varepsilon, Y)$ and $\E_1(v,Y)$ are shown in Figure~\ref{fig:FiboZ2}.

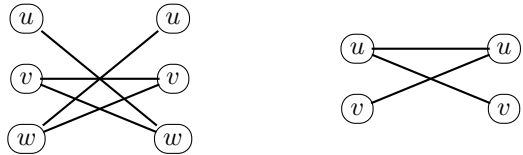
\begin{figure}[hbt]
	\centering
	\tikzset{node/.style={rectangle,draw,rounded corners=1.2ex}}
	\begin{tikzpicture}
	\node[node](u1l) {$u$};
	\node[node](v1l) [below= 0.4cm of u1l] {$v$};
	\node[node](w1l) [below= 0.4cm of v1l] {$w$};
	\node[node](u1r) [right= 1.5cm of u1l] {$u$};
	\node[node](v1r) [below= 0.4cm of u1r] {$v$};
	\node[node](w1r) [below= 0.4cm of v1r] {$w$};
	\path[draw,thick, shorten <=0 -1pt, shorten >=-1pt]
		(u1l) edge node {} (w1r)
		(v1l) edge node {} (v1r)
		(v1l) edge node {} (w1r)
		(w1l) edge node {} (u1r)
		(w1l) edge node {} (v1r);
	\node[node](u2l) [below right = 0cm and 2cm of u1r] {$u$};
	\node[node](v2l) [below= 0.4cm of u2l] {$v$};
	\node[node](u2r) [right= 1.5cm of u2l] {$u$};
	\node[node](v2r) [below= 0.4cm of u2r] {$v$};
	\path[draw,thick, shorten <=0 -1pt, shorten >=-1pt]
		(u2l) edge node {} (u2r)
		(u2l) edge node {} (v2r)
		(v2l) edge node {} (u2r);
	\end{tikzpicture}
	\caption{The graphs $\E_1(\varepsilon,Y)$ and $\E_1(w,Y)$.}
	\label{fig:FiboZ2}
\end{figure}
\end{example}

A particular case of complete bifix decoding is related to a notion which is well-known in topological dynamics, namely the skew product of two dynamical systems (see~\cite{CornfeldFominSinai1982}).
Indeed, assume that when we start with a permutation group $G$ on a set $Q$ and a morphism $f:A^*\rightarrow G$.
We denote $q \mapsto q \cdot w$ the result of the action of the permutation $f(w)$ on the point $q \in Q$.
Fix a point $i \in Q$.
The set of words $w$ such that $i\cdot w = i$ is a submonoid generated by a bifix code $U$ which is two-sided complete (this follows from \cite[Theorem 4.2.11]{BerstelDeFelicePerrinReutenauerRindone2012}).
The corresponding decoding is a shift space which is related to the skew product of $(X, \sigma)$ and $(G,Q)$.
It is the shift space $Y$ on the alphabet $A \times Q$ formed by the labels of the two-sided infinite paths on the graph with vertices $Q$ and edges $(p,q)$ labeled $(a,p)$ for $a \in A$ such that $p \cdot f(a)=q$.
The decoding of $X$ corresponding to $U$ is the dynamical system induced by $Y$ on the set of $y \in Y$ such that $y_0=(a,i)$ for some $a \in A$.

\begin{example}
\label{ex:FiboZ2}
Let $X$ be the Fibonacci shift, let $Q = \{ 1,2 \}$ and $G = \Z/2\Z$.
Let $f : A^* \rightarrow G$ be the morphism $a \mapsto (12), b \mapsto(1)$.
Choosing $i=1$, the bifix code $U$ build as above is $U = \{ aa,aba,b \}$ as in Example~\ref{ex:Fibobifix}.
\end{example}

\section{Conclusion}
\label{sec:conclusion}

The class of eventually dendric shifts is shown in this paper to have strong closure properties.
It leaves open the question of whether it is closed under taking \emph{factors}, that is, images by a sliding block code not necessarily bijective.

It would  be interesting to know how other properties which are known to hold for dendric shifts
extend to the this more general class.
This includes the following:
\begin{enumerate}
	\item To which extent the properties of return words proved for minimal dendric shifts extend to eventually dendric ones.
	For example, what can we say about the subgroup of the free group generated by return words to a given word?
	In~\cite{BertheDeFeliceDolceLeroyPerrinReutenauerRindone2015} it is proved that for minimal dendric sets, every set of return words is a basis of the free group, while in the case of specular sets, the set of return word to a fixed word is a basis of a particular subgroup called the even subgroup (see~\cite{BertheDeFeliceDelecroixDolceLeroyPerrinRindone2017}).
	\item Is there a finite $S$-adic representation for all minimal eventually dendric shifts?
	There is one for minimal dendric shifts~\cite{BertheDeFeliceDolceLeroyPerrinReutenauerRindone2013m}.
	\item Is the property of being eventually dendric decidable for a substitutive shift, as it is for dendric ones \cite{DolceKyriakoglouLeroy2016}?
\end{enumerate}

It would also be interesting to know whether the conjugacy of effectively given
eventually dendric shifts is decidable (the conjugacy of substitutive shifts was recently shown to be decidable~\cite{DurandLeroy2018}).

\bibliographystyle{plain}

\begin{thebibliography}{10}
	
	\bibitem{BalkovaPelantovaSteiner2008}
	{\soft{L}}ubom{\'{\i}}ra Balkov{\'a}, Edita Pelantov{\'a}, and Wolfgang
	Steiner.
	\newblock Sequences with constant number of return words.
	\newblock {\em Monatsh. Math.}, 155(3-4):251--263, 2008.
	
	\bibitem{BerstelDeFelicePerrinReutenauerRindone2012}
	Jean Berstel, Clelia De~Felice, Dominique Perrin, Christophe Reutenauer, and
	Giuseppina Rindone.
	\newblock Bifix codes and {S}turmian words.
	\newblock {\em J. Algebra}, 369:146--202, 2012.
	
	\bibitem{CodesAutomata}
	Jean Berstel, Dominique Perrin, and Christophe Reutenauer.
	\newblock {\em Codes and Automata}.
	\newblock Cambridge University Press, 2009.
	
	\bibitem{BertheDeFeliceDelecroixDolceLeroyPerrinRindone2017}
	Val\'erie Berth\'e, Clelia De~Felice, Vincent Delecroix, Francesco Dolce,
	Julien Leroy, Dominique Perrin, Christophe Reutenauer, and Giuseppina
	Rindone.
	\newblock Specular sets.
	\newblock {\em Theoret. Comput. Sci.}, 684:3--28, 2017.
	
	\bibitem{BertheDeFeliceDolceLeroyPerrinReutenauerRindone2015}
	Val\'erie Berth\'e, Clelia De~Felice, Francesco Dolce, Julien Leroy, Dominique
	Perrin, Christophe Reutenauer, and Giuseppina Rindone.
	\newblock Acyclic, connected and tree sets.
	\newblock {\em Monatsh. Math.}, 176(4):521--550, 2015.
	
	\bibitem{bifixcodesintervalexchanges}
	Val{\'e}rie Berth{\'e}, Clelia De~Felice, Francesco Dolce, Julien Leroy,
	Dominique Perrin, Christophe Reutenauer, and Giuseppina Rindone.
	\newblock Bifix codes and interval exchanges.
	\newblock {\em J. Pure Appl. Algebra}, 219(7):2781--2798, 2015.
	\newblock (http://dx.doi.org/10.1016/j.jpaa.2014.09.028).
	
	\bibitem{BertheDeFeliceDolceLeroyPerrinReutenauerRindone2013m}
	Val\'erie Berth\'e, Clelia De~Felice, Francesco Dolce, Julien Leroy, Dominique
	Perrin, Christophe Reutenauer, and Giuseppina Rindone.
	\newblock Maximal bifix decoding.
	\newblock {\em Dicrete Math.}, 338:725--742, 2015.
	
	\bibitem{Cassaigne1997}
	Julien Cassaigne.
	\newblock Complexit\'e et facteurs sp\'eciaux.
	\newblock {\em Bull. Belg. Math. Soc. Simon Stevin}, 4(1):67--88, 1997.
	\newblock Journ{\'e}es Montoises (Mons, 1994).
	
	\bibitem{CornfeldFominSinai1982}
	Isaac~P. Cornfeld, Sergei~V. Fomin, and Yakov~G. Sinai.
	\newblock {\em Ergodic theory}, volume 245 of {\em Grundlehren der
		Mathematischen Wissenschaften [Fundamental Principles of Mathematical
		Sciences]}.
	\newblock Springer-Verlag, New York, 1982.
	\newblock Translated from the Russian by A. B. Sosinskii.
	
	\bibitem{DolceKyriakoglouLeroy2016}
	Francesco Dolce, Revekka Kyriakoglou, and Julien Leroy.
	\newblock Decidable properties of extension graphs for substitutive languages.
	\newblock 2016.
	\newblock 15\`emes Journ\'ees Montoises d’informatique th\'eorique, Li\`ge
	(Belgique).
	
	\bibitem{DolcePerrin2017}
	Francesco Dolce and Dominique Perrin.
	\newblock Neutral and tree sets of arbitrary characteristic.
	\newblock {\em Theoret. Comput. Sci.}, 658(part A):159--174, 2017.
	
	\bibitem{DonosoDurandMaassPetite2016}
	Sebastian Donoso, Fabien Durand, Alejandro Maass, and Samuel Petite.
	\newblock On automorphism groups of low complexity subshifts.
	\newblock {\em Ergod. Th. Dynam. Sys.}, 36:64--95, 2016.
	
	\bibitem{DurandLeroy2018}
	Fabien Durand and Julien Leroy.
	\newblock Decidability of the isomorphism and the factorization between minimal
	substitution subshifts.
	\newblock 2018.
	\newblock \url{https://arxiv.org/abs/1806.04891}.
	
	\bibitem{DurandLeroyRichomme2013}
	Fabien Durand, Julien Leroy, and Gwena{\"e}l Richomme.
	\newblock Do the properties of an {$S$}-adic representation determine factor
	complexity?
	\newblock {\em J. Integer Seq.}, 16(2):Article 13.2.6, 30, 2013.
	
	\bibitem{PytheasFogg2002}
	N.~Pytheas Fogg.
	\newblock {\em Substitutions in dynamics, arithmetics and combinatorics},
	volume 1794 of {\em Lecture Notes in Mathematics}.
	\newblock Springer-Verlag, Berlin, 2002.
	\newblock Edited by V. Berth{\'e}, S. Ferenczi, C. Mauduit and A. Siegel.
	
	\bibitem{LindMarcus1995}
	Douglas Lind and Brian Marcus.
	\newblock {\em An Introduction to Symbolic Dynamics and Coding}.
	\newblock Cambridge University Press, Cambridge, 1995.
	
	\bibitem{Queffelec2010}
	Martine Queff{\'e}lec.
	\newblock {\em Substitution dynamical systems---spectral analysis}, volume 1294
	of {\em Lecture Notes in Mathematics}.
	\newblock Springer-Verlag, Berlin, second edition, 2010.
	
	\bibitem{Vesely2018}
	Vojtch Vesely.
	\newblock Properties of morphic images of {$S$}-adic sequences.
	\newblock Master's thesis, Czech Technical University, 2018.
	
\end{thebibliography}
\def\cprime{$'$}

\end{document}